\documentclass[11pt,dvips]{article}

\usepackage[english]{babel}
\usepackage{amssymb}
\usepackage{amsmath}
\usepackage{amsthm}
\usepackage{latexsym}
\usepackage{mathrsfs}
\usepackage{enumerate}
\usepackage{verbatim}
\usepackage{multirow}
\usepackage[dvips]{graphicx}		  
\usepackage{epstopdf}				  
\usepackage{xcolor}

\usepackage{fancyhdr}
\newtheorem{theo}{Theorem}
\newtheorem{prop}{Proposition}
\newtheorem{lemma}{Lemma}

\def\R{\mathbb{R}}

\def\E{\mathbb{E}}

\def\cC{\mathcal{C}}
\def\cO{\mathcal{O}}
\def\cI{\mathcal{I}}

\def\txtd{{\textnormal{d}}}
\def\txte{{\textnormal{e}}}

\renewenvironment{proof}[1]{%
	\par\vspace{1\baselineskip}%
	\noindent{\bf Proof#1\ \ }\ignorespaces }{%
	\nobreak\hfill\mbox{\ \ $\Box$}%
	\par\vspace{1\baselineskip}%
}

\newcommand{\bino}{\bigskip\noindent}

\newcommand{\dd}{\mbox{\rm\,d}}

\newcommand{\ddd}{\mbox{\,{\scriptsize \rm d}}}
\newcommand{\ra}{\rightarrow}
\newcommand{\dt}{\frac{\dd}{\dd t}}

\newcommand{\C}{\mathbb{C}}

\newcommand{\bd}{\begin{displaymath}}
\newcommand{\ed}{\end{displaymath}}
\newcommand{\be}{\begin{equation}}
\newcommand{\ee}{\end{equation}}
\newcommand{\benn}{\begin{equation*}}
\newcommand{\eenn}{\end{equation*}}
\newcommand{\bea}{\begin{eqnarray}}
\newcommand{\eea}{\end{eqnarray}}
\newcommand{\bda}{\begin{eqnarray*}}
\newcommand{\eda}{\end{eqnarray*}}
\newcommand{\ba}{\begin{eqnarray*}\begin{aligned}}
\newcommand{\ea}{\end{aligned}\end{eqnarray*}}
\newcommand{\ban}{\begin{eqnarray}\begin{aligned}}
\newcommand{\ean}{\end{aligned}\end{eqnarray}}

\normalsize  

\begin{document}

\title{Heterogeneous Population Dynamics and Scaling Laws near Epidemic Outbreaks}

\author{A. Widder \thanks{ORCOS, Institute of Mathematical Methods in Economics, Vienna, 
University of Technology, Wiedner Hauptstrasse 8, A-1040 Vienna, Austria, e-mail: 
{\tt andreas.widder@tuwien.ac.at}.}
\and
C. Kuehn\thanks{Institute for Analysis and Scientific Computing, Vienna, 
University of Technology, Wiedner Hauptstrasse 8-10, A-1040 Vienna, Austria, e-mail: 
{\tt ck274@cornell.edu}.}}

\date{\today}
\maketitle

\begin{abstract}
In this 
paper, we focus on the influence of heterogeneity and stochasticity of the population on the 
dynamical structure of a basic susceptible-infected-susceptible (SIS) model. First 
we prove that, upon a suitable mathematical reformulation of the basic reproduction number, the
homogeneous system and the heterogeneous system exhibit a completely analogous global
behaviour. Then we consider noise terms to incorporate the fluctuation effects and the 
random import of the disease into the population and analyse the influence of heterogeneity 
on warning signs for critical transitions (or tipping points). This theory shows that one may 
be able to anticipate whether a bifurcation point is close before it happens. We use numerical 
simulations of a stochastic fast-slow heterogeneous population SIS model and show various aspects 
of heterogeneity have crucial influences on the scaling laws that are used as early-warning 
signs for the homogeneous system. Thus, although the basic structural qualitative dynamical 
properties are the same for both systems, the quantitative features for epidemic prediction 
are expected to change and care has to be taken to interpret potential warning signs for disease 
outbreaks correctly.
\end{abstract}



\section{Introduction} 
\label{intro}

Infectious diseases have a big influence on the livelihood (and indeed lives) of individual people 
as well as the performance of whole economies \cite{O}. The development and understanding of 
mathematical models that can explain and especially predict the spreading of such diseases is 
therefore of enormous importance. A seminal work in this area was provided by Kermack and 
McKendrick in 1927 \cite{KM}. Up to this day their model is used as basis for analysis, although 
it has of course been extended in numerous ways. One such way is to consider heterogeneous populations. 
This is due to the realisation that individual people differ in their genetics, biology and social 
behaviour in ways that influence the spreading of infectious diseases. One type of model treats these 
individual traits as a static parameter \cite{CMLB, HR, N}. Since these parameters have a certain 
distribution amongst the population some information may be gained by studying the moments of 
this distribution \cite{D, V}. Other models deal with time varying heterogeneities like age or 
duration of the infection \cite{AI, FTV, I}. For a more complete overview of different ways to model 
heterogeneity in this context we refer to textbooks on mathematical epidemiology such 
as \cite{BDW, DH, KR}.

\medskip
In this paper we will exclusively deal with susceptible-infected-susceptible, in short SIS, 
models. These models assume that an individual is either infected or susceptible, and furthermore 
that an infected individual recovers from the infection with no lasting immunity and immediately 
becomes susceptible again. One of the main applications of SIS models are sexually transmitted 
diseases \cite{BC1,GA,HS,YHN}, but other bacterial infections can also be modelled this way \cite{H}. 
There are also applications of this model outside of biology, for example in the study of spreading 
of computer viruses \cite{KW,WM} or social contagions \cite{HRNC}. Heterogeneous versions of 
this model also have a long history, see for example \cite{LY}.\medskip

One feature that is present in most of these models is the existence of a threshold that fundamentally 
influences the behaviour of the system. This threshold is usually given in terms of the basic 
reproduction number $R_0$. If this number is smaller than one, then a disease can not lead to an 
outbreak and usually a disease free population rests, mathematically speaking, in a stable steady 
state. If $R_0$ is however bigger than one, then a disease can become endemic in a population. In 
many diseases this number is not constant but is susceptible to seasonal or other environmental 
changes \cite{ADH, KR}. Hence, it is important to provide rigorous mathematical analysis, how
$R_0$ has to be viewed for heterogeneous populations \cite{CHJ}.\medskip

Once we understand this influence of heterogeneity, then it is of great interest to analyse possible warning 
sings that indicate the approach of $R_0$ to the critical value, when $R_0$ depends upon parameters. One 
approach to model this setup is to consider epidemic dynamics as a multiple time scale system where 
the population dynamics, including infection and recovery, are fast while parameters influencing $R_0$ 
drift slowly, so that $R_0$ increases from the sub-threshold regime $R_0<1$ to the critical value $R_0=1$. 
Considering also stochastic perturbations, it has been shown in various epidemic models \cite{KuehnCT2,OD} 
that there exist warning signs for the upcoming critical value when recoding time series from the 
sub-threshold regime $R_0<1$. We follow in this vein and study how incorporating heterogeneity into 
a stochastic SIS model influences the warning sings of an impending critical transition.\medskip

Our two main results for dynamics and warning signs for heterogeneous SIS models can be summarized 
on a non-technical level as follows:

\begin{itemize}
 \item[(R1)] We prove a theorem, how the global dynamical structure of the deterministic (i.e.\ no noise)
 homogeneous population model is preserved when the homogeneous population is replaced by a heterogeneous one.
 In particular, the result shows that upon very reasonable modelling assumptions on the heterogeneity,
 the homogeneous and heterogeneous models have the same basic bifurcation structure with an epidemic threshold
 at $R_0=1$. 
 \item[(R2)] We extend the heterogeneous SIS model by stochastic perturbations as well as 
 by slow parameter dynamics. We use numerical simulations to investigate warning signs for epidemic outbreaks
 based upon scaling laws of the variance in the sub-threshold regime. We show that the rate of variance 
 can change below the epidemic threshold when
 \begin{enumerate}
 \item[(a)] a cut-off for the heterogeneities is considered, 
 \item[(b)] a discretise distribution of heterogeneities is considered, 
 \item[(c)] if the system interacts with the upper and lower level population boundaries,
 \item[(d)] if the transmission rate cannot be separated into a product of parametric drift and 
 contribution from heterogeneity.
 \end{enumerate}
 We also provide first steps to explain (a)-(d) on a non-rigorous level via formal calculations
 and considering the influences of various terms in the model.
\end{itemize}

The main implications for prediction and management of epidemic outbreaks are twofold. First, an 
epidemic threshold still exists for heterogeneous populations. It may shift due to the distribution of 
types in the heterogeneous population considered but we still have a tipping point or critical 
transitions towards an endemic state. This shows that there is a need to develop warning signs that can
be applied before the outbreak. However, warning signs from homogeneous population models do not
directly generalize to the heterogeneous situation. In particular, the functional form at which
the warning sign of variance rises, does depend crucially on many additional factors, which are 
not predicted by simple homogeneous SIS-models.\medskip 

The paper is structured as follows. In Section \ref{HomMod} we briefly review a homogeneous SIS-model 
as a baseline for our considerations. We state some known results about this model that are relevant 
to our analysis. In Section \ref{HetMod} we introduce the heterogeneous model we wish to study. In
Section \ref{Persist} we state, prove, and interpret the first main result (R1). In Section \ref{Extend}
we explain, why we extend the model by a slow parameter drift and by a noise term. Furthermore, we explain 
some background from the theory of warning signs for stochastic multiscale SIS models with homogeneous
populations. In Section \ref{NumRes}, we numerically analyse the influence that heterogeneity has on 
the behaviour of the system near bifurcation point by looking at the variance as warning sign. In Section
\ref{Expl}, we provide a few first steps to explain the numerical observations. In particular, Sections
\ref{NumRes}-\ref{Expl} provide the details for our second main result (R2). We conclude in Section 
\ref{Outlook} with an outlook of future problems for epidemic models with heterogeneous populations 
which arose during our analysis.\medskip

\textbf{Acknowledgements:} AW would like to thank the Austrian Science Foundation (FWF) for 
support under grant P 24125-N13. CK would like to thank the Austrian Academy of Science (\"OAW) 
for support via an APART Fellowship and the EU/REA for support via a Marie-Curie Integration
Re-Integration Grant. Both authors acknowledge very stimulating discussions with Vladimir Veliov 
at the beginning of this work.

\section{The homogeneous model} 
\label{HomMod}

Basic SIS-models are well understood and an in-depth discussion of them can be found in introductory 
books about mathematical epidemiology (e.g.\ \cite{BC,BDW,KR}). As a baseline homogeneous model we 
use
\ban \label{Eq_hom_mod}
	\dot{S}(t) &= 
	-\beta \frac{I(t)}{S(t)+I(t)}S(t) - \eta S(t) + \gamma I(t),\quad S(0)=S_0\geq 0,\\
	\dot{I}(t) &= 
	\beta \frac{I(t)}{S(t)+I(t)}S(t) + \eta S(t) - \gamma I(t), \quad I(0)=I_0\geq 0,
\ean
where $\frac{\txtd}{\txtd t}=\dot{~}$ denotes the time derivative, $\beta>0$ is the transmission rate 
and $\gamma>0$ the recovery rate. The parameter $\eta$ models the propagation of the disease due to imported 
cases of the infection. Such import can, for example, be explained by brief contacts with individuals outside 
of the population (see \cite{KR,OD}). The parameter $\eta$ has also been used with different interpretations. In \cite{HRNC} and \cite{HRNC2} $\eta$ denotes spontaneous self-infection in the transmission of social contagions. In \cite{SDC} and \cite{YCF} $\eta$ is a time dependent function modelling an infective medium. Furthermore, in \cite{ABDWZ} the mean field approximation of the $\epsilon-$SIS model introduced in \cite{VC} is presented. It too is an (heterogeneous) SIS-model with positive $\eta$.  The general model with $\eta>0$ will be used in the analysis of the steady states and bifurcation of the deterministic heterogeneous model. For the analysis of the stochastic model it will be included in the noise term.

 Note that the positive quadrant is invariant for \eqref{Eq_hom_mod}
so our choice of initial conditions ensures that the population sizes of infected and susceptibles remain
non-negative. We shall only consider $(I(t),S(t))\in[0,+\infty) \times [0,+\infty)$ for $t\geq 0$ from
now on.

\medskip
By adding the two equations in \eqref{Eq_hom_mod}, it is easy to see that $S(t)+I(t)$ is constant 
in this model. Because of the structure of \eqref{Eq_hom_mod} we can assume without loss of generality 
that $S(t)+I(t)=1$, since re-scaling both variables $S(t)$ and $I(t)$  by the inverse of the population 
size yields a total population of size $1$. By substituting $S(t)=1-I(t)$ into the equation for $I(t)$ we 
can describe the system by the single equation
\ban \label{Eq_hom_I}
	\dot{I}(t) &= \beta(1-I(t))I(t) + \eta(1-I(t)) - \gamma I(t),\quad I(0)=I_0.
\ean
If $\eta=0$ then we define $R_0=\frac{\beta}{\gamma}$, known as the basic reproduction number. 
If $R_0 \leq 1$ then (\ref{Eq_hom_I}) has single steady state, $I^*\equiv 0$, that is globally asymptotically 
stable. If $R_0>1$ then (\ref{Eq_hom_I}) has two steady states. The steady state $I^*\equiv0$ remains 
but is now unstable. The second steady state is $I^{**}\equiv \frac{\beta-\gamma}{\beta}$ and is globally 
asymptotically stable with the exception of $I_0=0$. From a mathematical perspective, this 
exchange-of-stability happens at a transcritical bifurcation when $(I,R_0)=(0,1)$. If $\eta>0$ 
then (\ref{Eq_hom_I}) always has one steady state. It is globally asymptotically stable, {i.e.}, all
non-negative initial conditions yield trajectories that are attracted in forward time to the steady state.

\section{The heterogeneous model} 
\label{HetMod}

We now modify the baseline model \eqref{Eq_hom_mod} by dividing the population according to some 
trait that is relevant to the spreading of the disease. This can indicate social behaviour like 
contact rates or biological traits like natural resistance towards the disease (for a detailed interpretation 
of heterogeneity we refer to introductory works in epidemiology, e.g.\ \cite{DH}). Each individual 
is assigned a heterogeneity state (h-state) $\omega$ which lies in some set $\Omega$. This $\omega$ 
can of course also be a vector carrying information about more than one trait.

\medskip
We assume that the disease spreads amongst the population of each h-state according to the dynamics
\ban\label{Eq_mod_het}
	\dot{S}(t,\omega) &= -\beta(\omega)\frac{J(t)}{T(t)+J(t)}S(t,\omega) - \eta(\omega) S(t,\omega) 
	+ \gamma(\omega) I(t,\omega),\quad S(0,\omega)=S_0(\omega)\\
	\dot{I}(t,\omega) &= \beta(\omega)\frac{J(t)}{T(t)+J(t)}S(t,\omega) + \eta(\omega) S(t,\omega)  
	- \gamma(\omega) I(t,\omega),\quad I(0,\omega)=I_0(\omega),
\ean
where we use the definitions
\benn
T(t):=\int_\Omega q(\omega)S(t,\omega)\dd\omega \qquad J(t):=\int_\Omega q(\omega)I(t,\omega)\dd\omega.
\eenn
Here we consider the following variables:
\begin{itemize}
\item $q(\omega)$ is the intensity of participation in risky interactions of an individual with 
h-state $\omega$,
\item $\beta(\omega)=\rho(\omega)q(\omega)$ where $\rho(\omega)$ is the force of infection of 
the disease towards an individual with h-state $\omega$,
\item $\gamma(\omega)$ is the recovery rate for an individual with h-state $\omega$,
\item $\eta(\omega)$ is the h-state dependent fraction of individuals that become infected 
through the import of the infection from outside the population.
\end{itemize}
Of course, $\eta(\omega)$ can also take any of the different interpretations mentioned in section \ref{HomMod}. Note that if all these variables are constant then the heterogeneous system (\ref{Eq_mod_het}) is 
equivalent to the homogeneous system (\ref{Eq_hom_mod}).\medskip

We want to make a short note about two aspects of this model. One is that for each $\omega$ the 
population $S(t,\omega)+I(t,\omega)$ is obviously constant. This implies that $\omega$ itself is 
not influenced by the disease and an individual that has h-state $\omega$ at the beginning remains 
in that h-state for the duration of our consideration. The second aspect is the transmission 
function $\frac{J(t)}{T(t)+J(t)}$. Transmission functions of this type have been used before 
\cite{D,FTV,V}. Such a transmission function can for example be derived by assuming a 
population with a heterogeneous social contact network \cite{N}. Models with such populations 
are at the centre of intensive current research (see e.g.\ \cite{BGM, BRPM, DHHE, KE}).

\medskip
We now formulate the mathematical assumptions for the heterogeneous population epidemic model 
used in our subsequent analysis. The set $\Omega$ is a complete Borel measurable space with a 
nonnegative measure $\mu$ and $\int_\Omega\dd\mu(\omega)=1$. All integration with respect to 
$\omega$ is taken to be with respect to that measure. All functions and parameters are assumed 
to be nonnegative and measurable with respect to $\mu$. Since $S(t,\omega)+I(t,\omega)$ is constant 
we can introduce a density function 
$f(\omega):=S(t,\omega)+I(t,\omega)$. We can assume without loss of generality that 
$\int_\Omega f(\omega)\dd\omega=1$. The function $q(\omega)$ is taken to be positive almost 
everywhere on $\Omega$, i.e.\ there is always some probability for risky interaction for each
h-state. We have
\ba
	T(t)+J(t) = \int_\Omega q(\omega)(S(t,\omega)+
	I(t,\omega))\dd\omega = \int_\Omega q(\omega)f(\omega)\dd\omega = C
\ea
for some constant $C>0$. By using  $\frac{q(\omega)}{C}$ instead of $q(\omega)$, we can assume 
without loss of generality that $T(t)+J(t)=1$. We also assume that the three functions 
$\beta(\omega)$, $\gamma(\omega)$ and $\eta(\omega)$ are bounded, which makes sense from a modelling
viewpoint. Furthermore, we assume there exists an $\varepsilon>0$ such that
\benn
\inf\limits_{\omega\in\Omega}\beta(\omega)\geq\varepsilon\quad  \text{and} 
\quad \inf\limits_{\omega\in\Omega}\gamma(\omega)\geq\varepsilon.
\eenn
These assumptions just mean that the transmission probability is never equal to zero when infected
and susceptible individuals meet and that there is always at least some positive, albeit potentially 
very long, time after which any infected individual recovers from the disease.  
An important consequence of these assumptions is that the functions $S(t,\cdot)$, $I(t,\cdot)$ are 
measurable for every $t\geq 0$ (see Theorem~1 in \cite{V2}).
For $\eta(\omega)$ we consider two cases. First the case that there exist a set $A\subseteq\Omega$ 
with positive measure such that $\eta(\omega)f(\omega)>0$ for $\omega\in A$. We denote this 
case by $\eta>0$. The second case where such a set does not exist will be denoted by $\eta=0$.

\medskip
Using $S(t,\omega)=f(\omega)-I(t,\omega)$ and $T(t)+J(t)=1$ we can describe the system 
(\ref{Eq_mod_het}) by
\ban\label{Eq_I}
	\dot{I}(t,\omega)&=(\beta(\omega)J(t)+\eta(\omega))f(\omega)-(\beta(\omega)J(t)+
	\eta(\omega)+\gamma(\omega))I(t,\omega),\quad I(0,\omega)=I_0(\omega),\\
	J(t)&=\int_\Omega q(\omega)I(t,\omega)\dd\omega.
\ean
It is now a natural question to ask which dynamical features are shared by the homogeneous population 
ordinary differential equation (ODE) given by \eqref{Eq_hom_I} and the heterogeneous population 
differential-integral equation \eqref{Eq_I}.

\section{Persistence of Dynamical Structure} 
\label{Persist}

In this section we show that in terms of steady state solutions and their stability properties 
the system (\ref{Eq_I}) exhibits the same behaviour as the system (\ref{Eq_hom_I}).

\begin{theo}\label{Th_stab}
If $\eta>0$ then the system (\ref{Eq_I}) has a unique steady state solution. This solution is 
globally asymptotically stable. If $\eta=0$ we define the basic reproduction number 
\ban\label{Eq_R0}
	R_0=\int_\Omega q(\omega)f(\omega)\frac{\beta(\omega)}{\gamma(\omega)}\dd\omega.
\ean
If $R_0\leq 1$ then (\ref{Eq_I}) has the unique steady state solution $I(t,\omega)=0$. This 
solution is globally asymptotically stable. If $R_0>1$ then (\ref{Eq_I}) has exactly two steady 
state solutions, one of which is $I(t,\omega)=0$. In this case, the solution $I(t,\omega)=0$ is 
an unstable 
steady state solution while the second steady state solution is globally asymptotically stable 
with the exception of $I_0(\omega)=0$ a.e.\ on $\Omega$.
\end{theo}

\begin{proof}{.}
We first show that the system does indeed have the number of steady states we claim it has. 
Let $\hat{I}(\omega)$ be a steady state of (\ref{Eq_I}) and 
$\hat{J}=\int_\Omega q(\omega)\hat{I}(\omega)\dd\omega$. As a steady state of (\ref{Eq_I}) 
$\hat{I}(\omega)$ is characterised by the equation
\ban\label{Eq_Ihat}
\hat{I}(\omega)=f(\omega)\frac{\beta(\omega)\hat{J}+\eta(\omega)}{\beta(\omega)\hat{J}
+\eta(\omega)+\gamma(\omega)}.
\ean
Plugging this into the equation for $\hat{J}$ yields
\ban \label{Eq_Jhat}
\hat{J}=\int_\Omega q(\omega)f(\omega)\frac{\beta(\omega)\hat{J}+\eta(\omega)}{\beta(\omega)\hat{J}+\eta(\omega)+\gamma(\omega)}\dd\omega.
\ean
Every solution $\hat{J}$ to (\ref{Eq_Jhat}) yields a steady state of (\ref{Eq_I}) by putting it into equation (\ref{Eq_Ihat}). Thus, we are searching for the roots of the function 
\ba
g(x)=\int_\Omega q(\omega)f(\omega)\frac{\beta(\omega)x+\eta(\omega)}{\beta(\omega)x+\eta(\omega)+\gamma(\omega)}\dd\omega-x
\ea
in the interval $[0,1]$. We have 
\ba
g(0)=\int_\Omega q(\omega)f(\omega)\frac{\eta(\omega)}{\eta(\omega)+\gamma(\omega)}\dd\omega
\ea
and
\ba
g(1)=\int_\Omega q(\omega)f(\omega)\frac{\beta(\omega)+\eta(\omega)}{\beta(\omega)+\eta(\omega)+\gamma(\omega)}\dd\omega-1< \int_\Omega q(\omega)f(\omega)\dd\omega-1=T(t)+J(t)-1=0.
\ea
A simple calculation yields
\ba
	g'(x)&=\int_\Omega q(\omega)f(\omega)\frac{\beta(\omega)\gamma(\omega)}{(\beta(\omega)x+\eta(\omega)+\gamma(\omega))^2}\dd\omega-1,\\
	g''(x)&=\int_\Omega q(\omega)f(\omega)\frac{-\beta(\omega)\gamma(\omega)2(\beta(\omega)x+\eta(\omega)+\gamma(\omega))\beta(\omega)}{(\beta(\omega)x+\eta(\omega)+\gamma(\omega))^4}\dd\omega.
\ea
Note that the second derivative is always negative. We consider the case $\eta>0$ first. We know 
that $g(x)=0$ has a solution since $g(0)>0$ and $g(1)<0$. Since $g(x)$ is concave this solution 
is unique.

Consider now the case $\eta=0$. In this case $g(0)=0$, so $0$ is a solution. If $g'(0)\leq 0$ then $g(x)$ negative on the whole interval $[0,1]$ due to the concavity of $g(x)$. If however $g'(0)>0$ then $g(x)$ is positive for small enough $x$. Using same reasoning as in the case $\eta>0$ we see that there exists a unique positive solution to $g(x)=0$. We therefore need to determine whether $g'(0)>0$. Since $\eta=0$ this is given by
\ba
g'(0)=\int_\Omega q(\omega)f(\omega)\frac{\beta(\omega)}{\gamma(\omega)}\dd\omega-1=R_0-1.
\ea
We see that if $R_0\leq1$ then $g'(0)\leq 0$ and $0$ is the only solution to $g(x)=0$, if $R_0>1$ then $g'(0)>0$ and there exists a unique solution in of $g(x)=0$ in $(0,1)$ alongside the solution $0$.

\bino
Now we want to show that the system converges to a steady state. In order to do this, we first need 
to show that $J(t)$ converges. In particular, we want to prove:

\begin{lemma}
\label{lem:convergence}
The limit $J^*=\lim\limits_{t\ra +\infty}J(t)$ exists. Furthermore,
\ban\label{Idot_sign}
\dot{I}(t,\omega)\lessgtr0\Leftrightarrow I(t,\omega)\gtrless \frac{ f(\omega)(\beta(\omega)J(t)+\eta(\omega))}{\beta(\omega)J(t)+\eta(\omega)+\gamma(\omega)}.
\ean
\end{lemma}

The proof of Lemma \ref{lem:convergence} is one major difficulty in this proof. However, the 
argument is quite lengthy and technical; hence we include it in Appendix \ref{ap:lemma}.

\bino
Now that we know that $J(t)$ converges it remains to show that if $\eta=0$ and $R_0>1$ then $J(t)$ converges a positive value and not to $0$ unless $I_0(\omega)=0$ a.e.\ on $\Omega$. Consider the inequality
\ba
\sup_{\zeta\in\Omega}\left(\frac{\beta(\zeta)}{\gamma(\zeta)}\right)J(t)=  \int_\Omega q(\omega)\sup_{\zeta\in\Omega}\left(\frac{\beta(\zeta)}{\gamma(\zeta)}\right)I(t,\omega)\dd\omega\geq \int_\Omega q(\omega)\frac{\beta(\omega)}{\gamma(\omega)}I(t,\omega)\dd\omega.
\ea
Thus, if $J(t)$ is positive and sufficiently small we have
\ba
&&R_0-1&>\int_\Omega q(\omega)\frac{\beta(\omega)}{\gamma(\omega)}I(t,\omega)\dd\omega\\
\Leftrightarrow&&J(t)R_0-J(t)&>\int_\Omega q(\omega)\frac{\beta(\omega)}{\gamma(\omega)}J(t)I(t,\omega)\dd\omega\\
\Leftrightarrow&&J(t)\int_\Omega q(\omega)f(\omega)\frac{\beta(\omega)}{\gamma(\omega)}\dd\omega-\int_\Omega q(\omega)I(t,\omega)\dd\omega &>\int_\Omega q(\omega)\frac{\beta(\omega)}{\gamma(\omega)}J(t)I(t,\omega)\dd\omega\\
\Leftrightarrow&&\int_\Omega \frac{q(\omega)}{\gamma(\omega)}\left(f(\omega)\beta(\omega)J(t)-\gamma(\omega)I(t,\omega)- \beta(\omega) J(t)I(t,\omega)\right)\dd\omega &>0\\
\Leftrightarrow&&\int_\Omega\frac{q(\omega)}{\gamma(\omega)}\dot{I}(t,\omega)\dd\omega&>0.
\ea
This shows that the term $\int_\Omega\frac{q(\omega)}{\gamma(\omega)}I(t,\omega)\dd\omega$ is monotonically increasing. But since
\ba
\int_\Omega\frac{q(\omega)}{\gamma(\omega)}I(t,\omega)\dd\omega\leq\frac{1}{\displaystyle\inf_{\omega\in\Omega}\gamma(\omega)}J(t),
\ea
we see that $J(t)$ is bounded below by a positive, monotonically increasing function. Therefore it can not converge to $0$. Since $I_0(\omega)>0$ on a set of positive measure we have that $J(0)>0$. Thus, $J(t)$ converges to a positive value.\\
Conversely, if $I_0(\omega)=0$ a.e.\ on $\Omega$ then $J(0)=0$. Directly from (\ref{Eq_I}) we see that in this case $\dot{I}(t,\omega)=0$ for a.e.\ $\omega\in\Omega$ and thus $J(t)=0$ for all $t\geq 0$.\\

Since $J(t)$ converges, the convergence of $I(t,\omega)$ follows immediately from (\ref{Idot_sign}). Obviously 
the limit of $I(t,\omega)$ is one of the steady states we identified above. We have also shown that if there are two steady states then $I(t,\omega)$ converges to the positive one, unless $I_0(\omega)=0$ a.e.\ on $\Omega$. Since in all other cases the convergence of $I(t,\omega)$ is independent of the initial data, the claim about asymptotic stability is proven. 
\end{proof}

\bino
In the case $\eta=0$ the value $R_0$ acts as a threshold value that determines whether there exists an endemic steady state or not. So far $R_0$ has only this mathematical meaning. The basic reproduction number is however a biological concept. Using the definition given in \cite{DHM}, the basic reproduction number is defined as the expected number of secondary cases produced, in a completely susceptible population, by a typical infected individual during its entire period of infectiousness. We now want to show that the value $R_0$ as we defined it coincides with this definition. Also in $\cite{DHM}$ the following result was obtained.
\begin{prop}\label{BRN}
Let $S(\omega)$ denote the density function of susceptibles describing the steady 
demographic state in the absence of the disease. Let $A(\tau,\zeta,\omega)$ be 
the expected infectivity of an individual which was infected $\tau$ units of time ago, 
while having h-state $\omega$ towards a susceptible which has h-state $\zeta$. Assume that
\ba
\int\limits_{0}^{\infty}A(\tau,\zeta,\omega)\dd\tau = a(\zeta)b(\omega).
\ea
Then the basic reproduction number $R_0$ for the system is given by
\ba
R_0=\int_\Omega a(\omega)b(\omega)S(\omega)\dd\omega.
\ea
\end{prop}

\bino
In our case the function $f(\omega)$ describes a steady state, provided that there are no 
infected individuals. The value $\beta(\omega)$ denotes the strength of infection for an individual 
with h-state $\omega$. The value $q(\omega)$ indicates the number of infectious contacts an 
infected individual with h-state $\omega$ has. On the other hand $\beta(\zeta)=\rho(\zeta)q(\zeta)$ 
is the average amount of risky contacts that would lead to an infection that an individual with 
h-state $\zeta$ has. The chance of an infectious contact between the infective $\omega$ individual 
and a specific $\zeta$ individual is therefore given by 
$q(\omega)\frac{\beta(\zeta)}{\int_\Omega q(\xi)f(\xi)\ddd\xi}=q(\omega)\beta(\zeta)$. In the 
absence of susceptible individuals the equation for the infected is given by $\dot{I}(t)=
-\gamma(\omega) I(t)$, which suggests that the probability that an infected individual is 
still infected at time $t$ is given by $\txte^{-\gamma(\omega) t}$. Since the infectivity of an 
individual is in our case independent of how long ago the individual was infected, we can 
conclude that the expected infectivity  $A(\tau,\zeta,\omega)$ is given by 
$q(\omega)\beta(\zeta)\txte^{-\gamma(\omega)\tau}$. Since 
\ba
\int\limits_{0}^{\infty}A(\tau,\zeta,\omega)\dd\tau=\int\limits_{0}^{\infty}q(\omega)\beta(\zeta)
\txte^{-\gamma(\omega)\tau}\dd\tau=\beta(\zeta)\frac{q(\omega)}{\gamma(\omega)},
\ea
we can use Proposition \ref{BRN} and get
\ba
R_0=\int_\Omega \beta(\omega)\frac{q(\omega)}{\gamma(\omega)}f(\omega)\dd\omega.
\ea
This is exactly the basic reproduction number as defined in Theorem \ref{Th_stab}.

\section{Extending the Model} 
\label{Extend}

Although the heterogeneous population SIS model \eqref{Eq_I} does capture additional realistic 
features of populations, there are several effects, which it cannot account for at all, or does
not account for very well. In particular, finite-size effects and small fluctuations are not included.
Furthermore, most realistic heterogeneous parameter distributions, e.g.\ the transmission rate, are 
not fixed in time but could be considered as additional dynamical variables. In this section, we 
extend the model \eqref{Eq_I} to include these effects.

\subsection{Noise}

In Section \ref{HomMod} we introduced several interpretations of the parameter $\eta$. Although it is modeled as a deterministic influence on the disease, the effect $\eta$ is supposed to describe on the other hand is seemingly of a 
random nature. Furthermore, even in a situation where we don't want to model any of these 
effects (i.e\ we set $\eta=0$) we can still expect there to be some random deviations from the transmission of the 
disease as predicted by the deterministic model. In fact, 
the validity of deterministic epidemiological models is usually argued by viewing them as the 
average transmission and recovery rate of individual random contacts in a sufficiently large 
population. It is therefore justified to expect to see some remaining randomness in the actual 
progression of the disease \cite{B,HK,LA,N1}.

\bino
We therefore want to model these random effects by exchanging the term 
containing $\eta$ with a term containing a stochastic process. A natural starting point 
for the case when the functional form and properties of the stochastic process are not known 
is to consider white noise $\xi=\xi(t)$ with mean zero $\E[\xi(t)]=0$ and $\delta$-correlation
$\E[\xi(t)-\xi(s)]=\delta(t-s)$, i.e.\ $\xi$ is a generalized stochastic process, so-called white noise, 
as discussed in \cite{ArnoldSDE}. We also want to consider the case when the noise depends upon the 
heterogenity and write $\xi=\xi(t,\omega)$ with the caveat that $\omega\in\Omega$ still denotes the 
variable measuring the heterogeneity distribution, while we suppress the underlying probability space 
for the stochastic process $\xi$ in the notation.\medskip

Putting $\xi(t,\omega)$ into equation (\ref{Eq_mod_het}) yields
\ban\label{Eq_mod_stoch}
	\dot{S}(t,\omega)&=-\beta(\omega)\frac{J(t)}{T(t)+J(t)}S(t,\omega)+\gamma(\omega) 
	I(t,\omega)-\sigma(\omega)\xi(t,\omega),\quad S(0,\omega)=S_0(\omega)\\
	\dot{I}(t,\omega)&= \beta(\omega)\frac{J(t)}{T(t)+J(t)}S(t,\omega)-\gamma(\omega) 
	I(t,\omega)+\sigma(\omega)\xi(t,\omega),\quad I(0,\omega)=I_0(\omega).
\ean
The function $\sigma(\omega):\Omega\ra [0,+\infty)$ is assumed to be bounded and basically provides 
the noise level for a specific $\omega$. Note 
that for every $\omega$, the sum $S(t,\omega)+I(t,\omega)$ is still constant. We can therefore 
again describe the system (\ref{Eq_mod_stoch}) by the smaller system
\ban\label{Eq_Istoch_add}
\dot{I}(t,\omega)&=\beta(\omega)J(t)f(\omega)-(\beta(\omega)J(t)+\gamma(\omega))I(t,\omega)
+\sigma(\omega)\xi(t,\omega),\\
J(t)&=\int_\Omega q(\omega)I(t,\omega)\dd\omega.
\ean
One problem with using an additive noise term is that $I(t,\omega)$ always has to be positive. 
But in this model it would be possible for $I(t,\omega)$ to become negative. To disallow this 
we will use
\ba
\dot{I}(t,\omega)=\max\{0,\beta(\omega)J(t)f(\omega)+\sigma(\omega)\xi(t,\omega)\}
\quad\text{if}\quad I(t,\omega)=0.
\ea
Similarly, since $I(t,\omega)$ has to be smaller than $f(\omega)$, we use
\ba
\dot{I}(t,\omega)=\min\{0,-\gamma(\omega)f(\omega)+\sigma(\omega)\xi(t,\omega)\}
\quad\text{if}\quad I(t,\omega)=f(\omega).
\ea
In the following considerations we restrict ourselves to models using additive noise. 
However, we want to indicate another commonly encountered modelling possibility, which is using 
a multiplicative noise term instead of an additive one. That is, to use
\ban\label{Eq_Istoch_mult}
\dot{I}(t,\omega)&=\beta(\omega)J(t)f(\omega)-(\beta(\omega)J(t)+\gamma(\omega))I(t,\omega)
+g(I(t,\omega),\omega)\xi(t,\omega),\\
J(t)&=\int_\Omega q(\omega)I(t,\omega)\dd\omega,
\ean
where $g:\R\times \Omega\ra [0,+\infty)$ is bounded. Imposing the conditions $g(0,\omega)=0$ and 
$g(1,\omega)=0$ can now ensure that it is never possible 
for $I(t,\omega)$ to become negative or larger than $f(\omega)$. Usually, one also assumes
that $g(\cdot,\omega)$ does not vanish between zero and one.\medskip

Which of these two options is chosen will depend on what kind of random influences are to be 
considered. If the random fluctuations are meant to offset fluctuations in the transmission 
and recovery of the infection, then the multiplicative noise term might be more appropriate. 
First of all, if no infected individuals are present then the disease does not spread at all, 
which is captured by this model. Also, if nearly no one (or nearly everyone) is infected then 
the inaccuracies of the deterministic model should be small, so the noise term should also 
be small. Again, the multiplicative noise exhibits this behaviour.

The model with additive noise, which we will use in the following, is however not without merit. 
It allows us to model a population that has contact with an outside source that can import the 
disease into the population. This source can be, as mentioned above, another population which 
imports the disease. Alternatively, there might be factors in the environment that import the 
infection. For a population of animals it could for example model the possibility to become
infected through one of its food sources. Also for human populations this allows us to assume 
that there are vermin or insects in their environment, which are carriers of the disease and 
are able to transmit it to humans. In these situations there is a chance to become infected 
even in a population that consists entirely of susceptible individuals, which is not captured 
in models using multiplicative noise. 

Also, it should be noted that if both effects are present, i.e.\ internal fluctuations as
well as external fluctuations, and we assume that both noise terms act as summands in the model,
then the noise term is
\be
\label{eq:nterm}
[\sigma(\omega)+g(I(t,\omega),\omega)]\xi(t,\omega).
\ee
Near the two states $I(t,\omega)\equiv 0$ and $I(t,\omega)\equiv f(\omega)$, we have that
$g(I(t,\omega),\omega)$ is a higher-order term in comparison to the constant term as long as
the constant term does not vanish and we are mainly interested in the regimes near the 
the two states $I(t,\omega)\equiv 0$ and $I(t,\omega)\equiv f(\omega)$ in the remaining part
of this work. Also note that a multiplicative noise term $g(I(t,\omega),\omega)\xi(t,\omega)$ with $g(0,\omega)>0$ can always be written as
\ba
g(I(t,\omega),\omega)\xi(t,\omega)= \left[g(0,\omega)+\big(g(I(t,\omega),\omega)-g(0,\omega)\big)\right]\xi(t,\omega),
\ea
which is near $I(t,\omega)=0$ again the sum of an additive noise term and a term of higher order. Based on these arguments, we proceed with additive noise
but it could definitely be interesting to investigate the purely multiplicative noise in future work.

\subsection{Multiple time scales}

As a final extension of our model we now introduce a slow variable into the system. Making 
certain model parameters slow dynamic variables is a very natural extension used in virtually
all areas of research in mathematical biology \cite{H1,KuehnBook}. The main reason is that it is
usually not correct to assume that all system parameters are fixed but most system parameters
are going to change slowly over time, so a parametric model should rather be viewed as a partially
frozen state for a model with multiple time scales.\medskip

In the context of epidemiology, many diseases have seasonal cycles or are latent for a longer 
period before it comes to an outbreak. In both cases we assume that the basic reproduction 
number $R_0$ was smaller than $1$ until some time, which means the stable steady state of 
the deterministic system is $0$, and bigger than $1$ afterwards, which 
means that a stable endemic steady state exists. In order to capture this in our model we assume 
that $\beta(\omega)$ slowly changes over time. In fact, there are many different possibilities
that may lead to a slowly changing transmission rate, including seasonal changes, evolutionary
processes, socio-economic influences, and so on. Furthermore, if we would keep the transmission
rate fixed as a parameter, then we would either observe a disease-free state or an endemic state
in the SIS model but not the transition between the two cases. It is precisely the \emph{dynamic}
transition regime which we are interested in.\medskip

We assume that the time dependence of the
function $\beta(t,\omega)$ is such that $R_0$ is increasing in $t$. For example, assume that 
$\beta(t,\omega)$ is separable, i.e.\ there exists a function $\beta_0(t)$ such that 
$\beta(t,\omega)=\beta_0(t)\beta(\omega)$, and that this function $\beta_0(t)$ evolves according 
to the equation $\dot{\beta_0}(t)=\varepsilon$ for $0<\varepsilon\ll 1$. In this case we would have
\ba
R_0(t)=\int_\Omega q(\omega)f(\omega)\frac{\beta_0(t)\beta(\omega)}{\gamma(\omega)}\dd\omega
=\beta_0(t)\int_\Omega q(\omega)f(\omega)\frac{\beta(\omega)}{\gamma(\omega)}\dd\omega.
\ea
Thus, $R_0(t)$ is strictly increasing and, if $\beta_0(0)$ is small enough, $R_0(0)<1$. 
This is exactly the situation we want to capture. This effect can of course also be achieved 
with a $\beta(t,\omega)$ which is not factorisable. We are therefore looking at the system
\ban\label{Eq_sys_add}
\dot{I}(t,\omega)&=\beta(t,\omega)J(t)(f(\omega)-I(t,\omega))-\gamma(\omega)
I(t,\omega)+\sigma(\omega)\xi(t,\omega),\\
\dot{\beta}(t,\omega)&=\varepsilon h(t,\omega),\\
J(t)&=\int_\Omega q(\omega)I(t,\omega)\dd\omega,
\ean
with an appropriate function $h(t,\omega)$.

\subsection{Warning-Signs for the Homogeneous Fast-Slow Stochastic Model}
\label{sec:warning}

In this section, we briefly recall some techniques for fast-slow systems and warning signs for
stochastic fast-slow systems. For reviewing this material, we consider a simple homogeneous 
version of \eqref{Eq_sys_add} to simplify the exposition
\be
\label{eq:fs}
\begin{array}{lcl}
 \dot{I}(t) &=& \beta(t) I(t)(1-I(t))-\gamma I(t)+\sigma \xi(t),\\
 \dot{\beta}(t) &=& \varepsilon, 
\end{array}
\ee
where $I=I(t)$ is the fast variable and $\beta=\beta(t)$ the slow variable. For $\sigma=0$, 
$\varepsilon=0$, the set $\cC_0=\{(I,\beta)\in[0,+\infty)\times[0,+\infty):I(\beta-\gamma-\beta I)=0\}$ 
is called the critical manifold \cite{Jones} and consists of steady states for the fast subsystem, 
which is obtained by setting $\varepsilon=0$ in \eqref{eq:fs}. The transcritical bifurcation discussed in 
Section \ref{HomMod} separates $\cC_0$ into three parts in the positive quadrant
\benn
\cC_0^a=\cC_0\cap \{R_0\leq1\},\quad \cC_0^r=\cC_0\cap \{R_0>1,I=0\},\quad \cC_0^e=\cC_0\cap \{I>0\}.  
\eenn
Then $\cC_0^a$ and $\cC_0^e$ consist of attracting steady states for 
the fast subsystem, while $\cC_0^r$ is repelling. The stability is exchanged at the transcritical
bifurcation point with $\beta=\gamma$, i.e.\ at $R_0=1$. The deterministic fast-slow systems 
analysis of the dynamic transcritical bifurcation with $0<\varepsilon\ll1$ can be found in 
\cite{KruSzm4,Schecter}, where one key point is that one can extend a perturbation $\cC_\varepsilon^a$,
a so-called attracting slow manifold, of $\cC_0^a$ up to a region of size $I\sim \cO(\varepsilon^{1/2})$
and $\beta-\gamma \sim \cO(\varepsilon^{1/2})$ as $\varepsilon\rightarrow 0$ near the transcritical
bifurcation point. The relevant conclusion for us here is that a linearisation analysis is expected
to be valid up to this region, excluding a small ball of size $\cO(\varepsilon^{1/2})$. 

It can be shown that sample paths of the stochastic system with $0<\varepsilon\ll1$, $0<\sigma\ll1$
also track with high-probability the attracting manifold inside a neighbourhood of order 
$\cO(\varepsilon)$ plus a probabilistic correction term \cite{BerglundGentz1}. However, as the 
transcritical bifurcation point $(I,\beta)=(0,\gamma)$ is slowly approached from below 
$\beta\nearrow \gamma$, the probabilistic correction term starts to grow. Indeed, there is a 
simple intuitive explanation for this behaviour due to an effect also called ``critical slowing 
down". To understand this effect, Taylor expand the drift and diffusion terms of the $I$-component
of the stochastic differential equation \eqref{eq:fs} around $\cC_0^a$ and keep the linear terms, 
which yields
\be
\label{eq:fs1}
 \dot{\cI}(t) =[\beta -\gamma] \cI(t)+\sigma \xi(t),
\ee
where we view $\beta$ as a parameter for now and use $\cI$ to emphasize that we work on the
level of the linearization. Then \eqref{eq:fs1} is just an Ornstein-Uhlenbeck (OU)
process \cite{Gardiner}. Consider the regime $\beta\leq \gamma$, then the variance of the OU process
increases if $\beta$ increases and it is an explicit calculation \cite{KuehnCT2} to see that
\be
\lim_{t\ra +\infty}\text{Var}(\cI(t))\sim\frac{\sigma^2}{\gamma-\beta}\qquad 
\text{as $\beta\nearrow \gamma$},
\ee
so the variance increases rapidly as we start to approach the bifurcation point by changing the parameter
$\beta$ more towards $\gamma$. This makes sense intuitively as the deterministic stabilizing effect
from the drift term $[\beta -\gamma] \cI(t)$ pushing towards a region near $\cC_0^a$ is diminished
(``critical slowing down") and hence the noisy fluctuations increase. It is known that the effect of 
critical slowing-down in combination with noise can be exploited to predict bifurcation points in certain 
situations (see e.g.\ the ground-breaking work \cite{Wiesenfeld1}). The idea has been also suggested
in the context of ecology \cite{CarpenterBrock} and then applied in many other circumstances 
\cite{Schefferetal}. In fact, one may prove that we indeed have for the full nonlinear stochastic 
fast-slow system \eqref{eq:fs}, under suitable smallness assumptions on a fixed noise level and staying 
$\cO(\varepsilon^{1/2})$ away from the region of the deterministic bifurcation point, that
\be
\label{eq:slaw}
\text{Var}(I(t))\sim \frac{A}{(t_{crit}-t)^\alpha}+\text{higher-order terms},\quad \text{as $t\nearrow t_c$},
\ee
where $\alpha=1$, $A=\sigma^2$, $\beta(t_{crit})=\gamma$ with $\beta(0)<t_{crit}$; the details can be found in
\cite{KuehnCT2} using moment expansion methods, and in \cite{BerglundGentz} using martingale methods and/or
explicit OU-process results. The main practical conclusion is that there is a leading-order scaling law 
of the variance as the value of $R_0$ is approached by letting the transmission rate slowly drift in time.
This scaling law can be used for prediction as the scaling exponent $\alpha=1$ is universal for a 
non-degenerate transcritical bifurcation. In fact, a calculation of the leading-order
covariance scaling laws for all bifurcations up to codimension-two in stochastic fast-slow systems
has been carried out \cite{KuehnCT2}, which builds a mathematical framework for generic systems.

However, this theory simply does not apply to the heterogeneous population model \eqref{Eq_sys_add}
we consider here. In particular, the influence of heterogeneity on early-warning signs has not been
investigated much to the best of our knowledge (for examples see \cite{KMC,Schefferetal}). Since it is a key effect in realistic models of
disease spreading, it is natural to ask, how it influences the scaling law \eqref{eq:slaw}.

\section{Numerical results} 
\label{NumRes}

Here we present numerical simulations of the homogeneous and heterogeneous system to see the influence 
of the heterogeneity. We have chosen to consider the variance as the early warning sign. We assume 
that the variance behaves like $\frac{A}{\left(t_{crit}-t\right)^\alpha}$ for appropriate $A$ and 
$\alpha$, where $t_{crit}$ is the time at which $R_0(t)=1$. The main difficulty lies in correctly 
determining $\alpha$. We will calculate it by fitting the reference curve 
$\frac{A}{\left(t_{crit}-t\right)^\alpha}$ to the time series of our simulation using 
the least squares method. To smoothen the time series we will average the variance over $100$ 
simulations. However, since the reference curve goes to infinity at $t_{crit}$ we do not fit the 
curve over the whole interval, which also takes into account the theory, which excludes a small 
$\varepsilon$-dependent ball near $R_0(t)=1$ as discussed in Section \ref{sec:warning}. We therefore calculate the best fit over $80\%$ or $90\%$ of the 
considered time interval. Generally, fitting over $90\%$ gives better results. In the cases we consider only $80\%$ of the interval, fitting over a larger part would not yield reasonable results as the solution goes to $-\infty$. These are cases in which the sample path drops below the negative unstable branch of the transcritical bifurcation (see Figure \ref{figure_9}). Since 
different choices in the size of the considered interval lead to slightly different values for $\alpha$ we cannot claim to calculate 
the exact $\alpha$ that the variance of $I(t)$ follows. We will however be able to detect changes 
in the level of $\alpha$ that are due to influences of the heterogeneity.\medskip

One further aspect we fix for all our considerations is the order in which we aggregate 
the system and calculate the variance. We could calculate the variance of $I(t,\omega)$ and then 
aggregate these variances, or first calculate $I(t)=\int_\Omega I(t,\omega)\dd\omega$ and calculate 
the variance of $I(t)$. We choose the latter option since in applications it is more feasible to 
be able to track the changes of the prevalence of the disease in the whole population rather 
than being able to track it for each h-state, as would be required by the first method.\medskip

In this section, we shall only consider the numerical simulations make observations about the
results. A more detailed discussion \emph{why} certain effects may occur is then given in
Section \ref{Expl}.

\bino
First we consider the homogeneous system with additive noise and a very simple multiplicative 
time dependency of $\beta$:
\ba
	\dot{I}(t) &= \beta~\beta_0(t)(1-I(t))I(t) - \gamma I(t)+\sigma \xi(t),\\
	\dot{\beta_0}(t)&=\varepsilon.
\ea
As initial conditions we choose $I(0)=0$ and $\beta_0(0)=0$. The parameters are chosen as 
$\beta=0.3$, $\gamma=0.4$, and $\sigma= 0.01$. The time scale separation parameter $\varepsilon$ 
for the slow variable drift is set to $\varepsilon=0.0001$. If we allow $I(t)$ to become negative 
then we know that 
the variance of $I(t)$ should behave as $\frac{A}{t_{crit}-t}$. In Figure \ref{figure_1} we 
show the variance of $I(t)$, averaged over $100$ simulations, and the reference curve with both the
theoretical exponent $\alpha=1$ and with the exponent provided by the best fit over $80\%$ of 
the time interval. The measured exponent is reasonably close to the theoretically predicted
value $\alpha=1$; this slight underestimate is expected as a transcritical bifurcation splits 
into two saddle bifurcations upon generic perturbations and for saddle-nodes the exponent 
is $\alpha=\frac12$; see also \cite{KuehnCT2} for more details, which exponents may occur
in the generic cases.\medskip

\begin{figure}
\centering
\includegraphics[width=\textwidth]{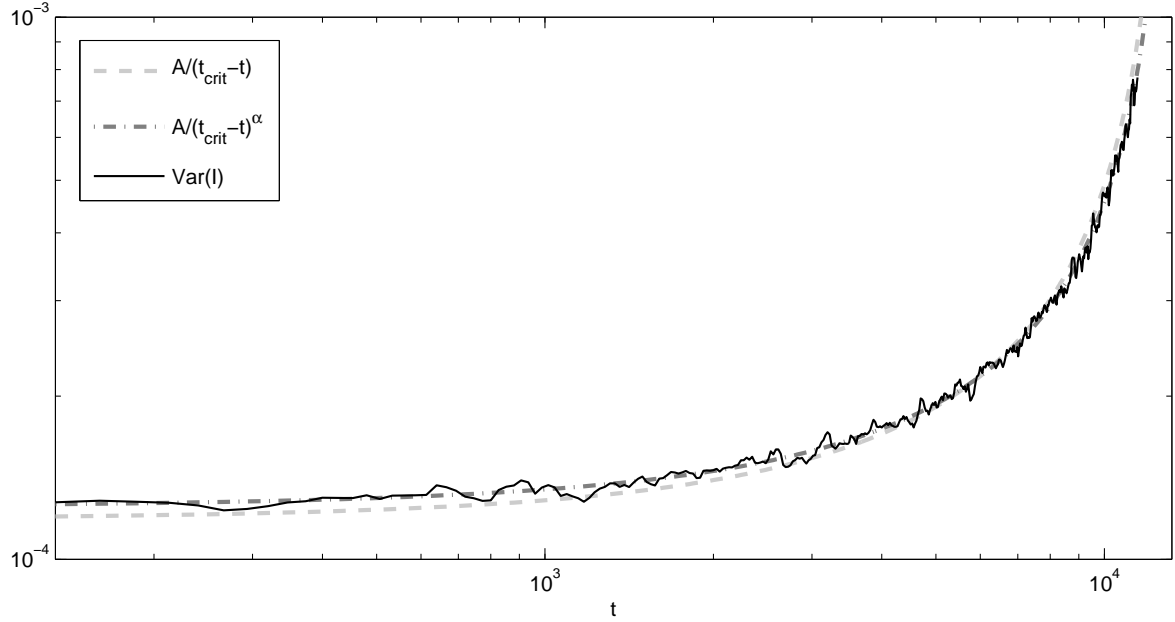}
\caption{Here we see the variance of the aggregated variable $I(t)$. As a reference we show the curve $A/(t_{crit}-t)^\alpha$ with both the expected theoretical exponent $\alpha = 1$ and with the exponent $\alpha = 0.9125$ provided by the best fit over $80\%$ of the considered time interval.}\label{figure_1}
\end{figure}


\begin{figure}
\centering
\includegraphics[width=\textwidth]{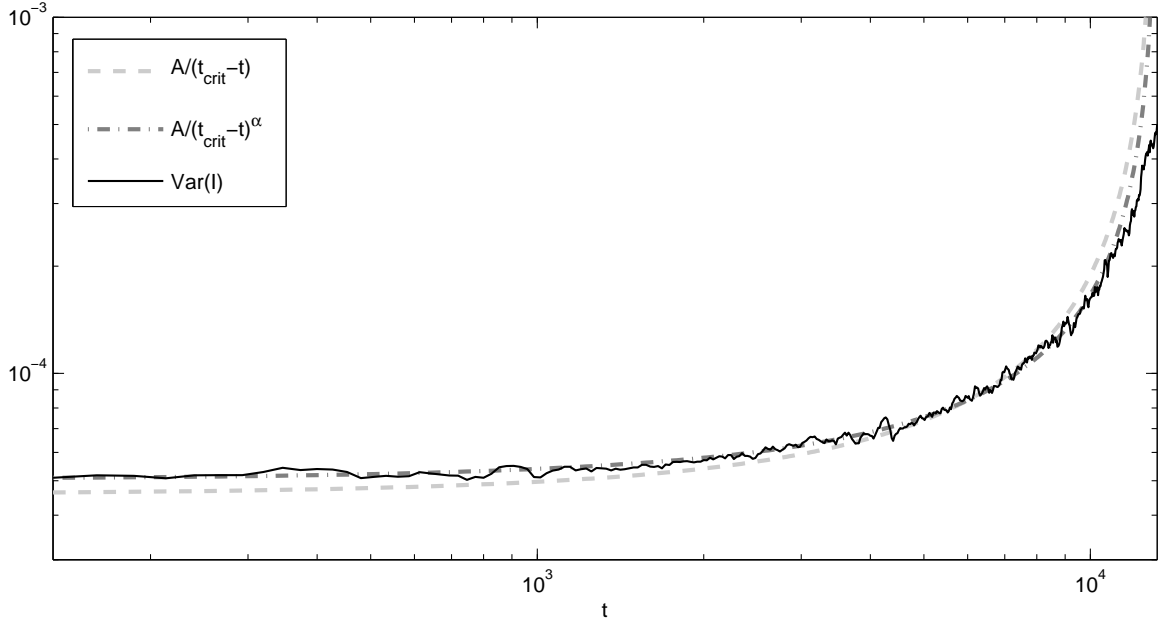}
\caption{Here we see the variance of the aggregated variable $I(t)$. As a reference we show the curve $A/(t_{crit}-t)^\alpha$ with both the expected theoretical exponent $\alpha = 1$ and with the exponent $\alpha = 0.8414$ provided by the best fit over $90\%$ of the considered time interval. Note that as we approach the critical moment the curve for the variance is noticeable below the curve with $\alpha=1$.}\label{figure_2}
\end{figure}

Figure \ref{figure_2} shows the result of this calculations if we cut off $I(t)$ at $0$, i.e.\ we 
use the rule
\ba
\dot{I}(t)=\max\{\sigma\xi(t),0\},\quad\text{if}\quad I(t)=0
\ea
for the discrete-time numerical scheme; for an introduction to numerical schemes for stochastic 
ordinary differential equations see \cite{Higham}. The results show that the key exponent 
$\alpha$ decreases in comparison to the system without cut-off.\medskip 

Next, we consider the heterogeneous system. We are going to consider situations in which the 
white noises $\xi(t,\omega)$ are dependent on each other for different $\omega\in\Omega$, or where 
the space of h-states is discrete to understand, which implications these assumptions have on the
model. Note that both assumptions have a direct modelling motivation. Usually, we may group or
cluster different parts of a heterogeneous population into different classes, e.g.\ all parts
with a different trait. Secondly, $\xi(t,\omega)$ models all stochastic internal and external 
effects and one natural assumption would be that all classes of the heterogeneous population
are subject to the same external fluctuations, which would lead to the case $\xi(t,\omega)=\xi(t)$, 
i.e.\ the same white noise acts on all h-states. Note that for the aggregated variable $I(t)$ we have
\ba
\dot{I}(t)&=\int_\Omega\beta(t,\omega)J(t)(f(\omega)-I(t,\omega))-\gamma(\omega)I(t,\omega)\dd\omega+\int_\Omega\sigma(\omega)\xi(t,\omega)\dd\omega.
\ea
If $\Omega$ is continuous and the $\xi(t,\omega)$ are independent of each other, then $\int_\Omega\sigma(\omega)\xi(t,\omega)\dd\omega=0$ and the influence of the noise is reduced to indirect effects. We therefore consider either continuous $\Omega$ with dependent $\xi(t,\omega)$ or a discrete $\Omega$ with independent $\xi(t,\omega)$.\medskip

\begin{figure}
\centering
\includegraphics[width=\textwidth]{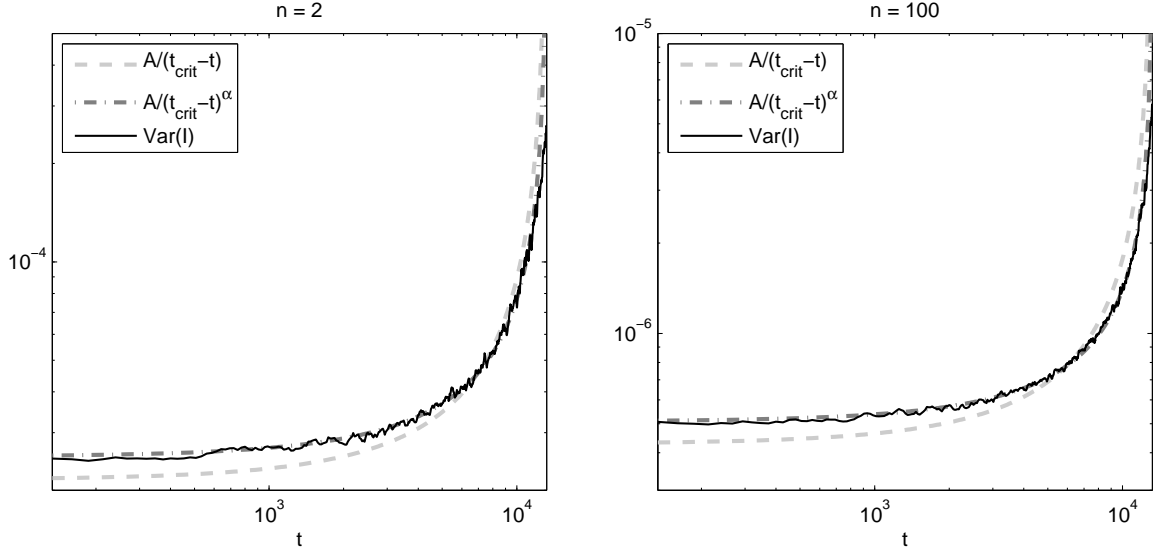}
\caption{The results for the discrete heterogeneous system for two different values of $n$. For $n=2$ the best fit results in $A=0.0432$ and $\alpha=0.7842$, for $n=100$ in $A=0.0004$ and $\alpha=0.7135$. Note that both values decrease for bigger $n$.}\label{figure_3}
\end{figure}

We start with the discrete h-state scenario. For an integer $n>1$ we set 
\benn
\Omega=\left\{\frac{i}{n-1} :i=0,\cdots,n-1 \right\}. 
\eenn
As measure $\mu$ we choose the counting measure normed to $1$ over $\Omega$
\benn
\int_\Omega \phi(\omega)\dd\omega=\frac{1}{n}\sum\limits_{i=1}^{n} \phi(\omega_i). 
\eenn
We assume that $\beta(t,\omega)=\beta_0(t)\beta(\omega)$ and $\dot{\beta_0}(t)=\varepsilon.$ As in 
the homogeneous case we choose $I(0)=0$ and $\beta_0(0)=0$ as initial conditions and $\beta=0.3$, 
$\gamma=0.4$, and $\sigma= 0.01$ for the parameters. The time scale separation parameter $\varepsilon$ 
for the slow variable is again set at $\varepsilon=0.0001$. Here the heterogeneity influences the number of elements in $\Omega$ and the distribution $f(\omega)$. Furthermore, $\xi(t,\omega)$ are chosen as $n$ independent identically distributed random variables.
We will both now and for continuous $\Omega$ later consider the distribution
\ba
f(\omega)=\frac{\frac{1}{\sqrt{2\pi}\theta}
\txte^{-\frac{(\omega-0.5)^2}{2\theta^2}}}{\displaystyle
\int\limits_{\Omega}\frac{1}{\sqrt{2\pi}\theta}\txte^{-\frac{(\zeta-0.5)^2}{2\theta^2}}\dd\zeta}.
\ea
This is simply a normal distribution with mean $0.5$ truncated to $\Omega$. Figure \ref{figure_5} 
shows $f(\omega)$ for different values of $p$. The parameter $\theta$ is the standard deviation of this
 distribution. Note that as $\theta$ goes towards 
$0$, the function $f(\omega)$ converges to the delta-distribution $\delta(\omega-0.5)$. Hence, the 
the heterogeneous system starts to approximate the homogeneous one as $\theta\rightarrow0$. On 
the other hand, if $\theta\rightarrow +\infty$ then $f(\omega)$ converges to the constant function 
$f(\omega)=1$. We will therefore parametrise $f(\omega)$ with 
$\theta=\frac{1}{(2p-2)^2}-\frac{1}{4}$ for $p\in(0,1)$. In the discrete case which we consider 
first, this yields approximately a binomial-type distribution. In Figure \ref{figure_3} we show 
the result for $p=0.5$ and two different choices of $n$. Figure \ref{figure_4} shows how both
 $\alpha$ and $A$ in the best fit curve change with increasing $n$. A clear trend
is observed showing that $\alpha$ (and $A$) decrease as $n$ is increased.\medskip 

\begin{figure}
\centering
\includegraphics[width=\textwidth]{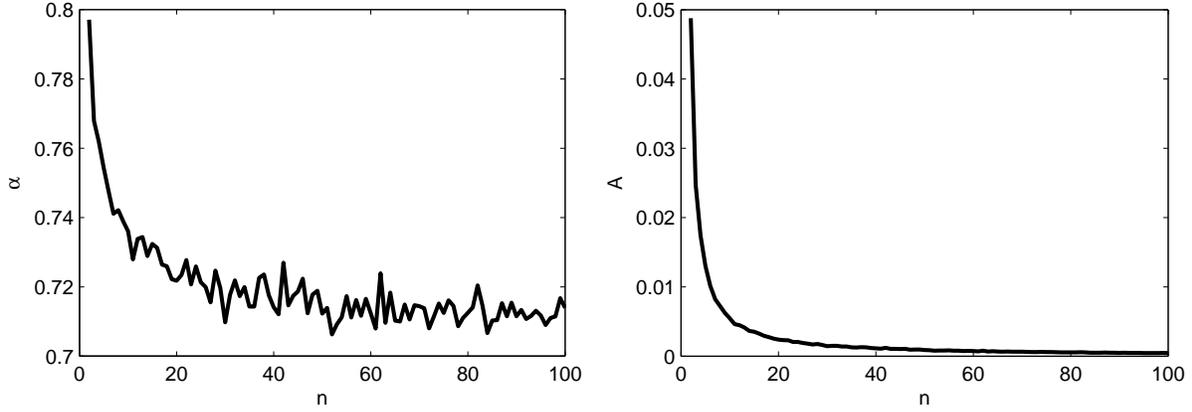}
\caption{We see the influence of $n$ on the on the parameters for the best fit, calculated over $90\%$ of the time interval. Both $\alpha$ and $A$ decrease as $n$ increases. The decrease is steep for small $n$ and approaches a constant level as $n$ becomes large.}\label{figure_4}
\end{figure}

\begin{figure}
\centering
\includegraphics[width=0.6\textwidth]{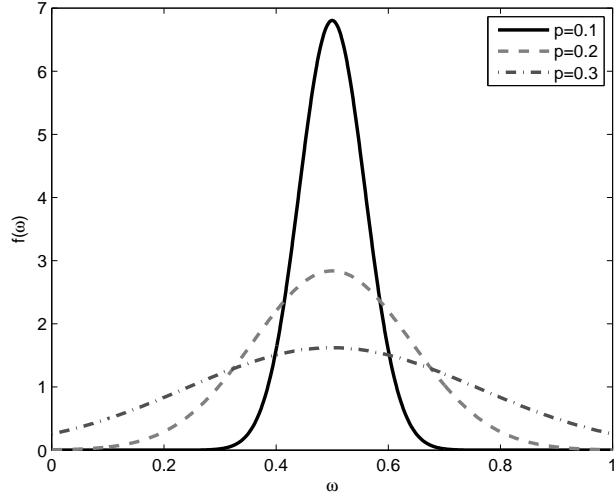}
\caption{Plot of $f(\omega)$ for different values of $p$. For small $p$ the function $f(\omega)$ approaches a $\delta$-Distribution at $0.5$. For larger $p$ the function becomes more flat. Note that for small $p$ the support of $f(\omega)$ increases with $p$.}\label{figure_5}
\end{figure}

\begin{figure}
\centering
\includegraphics[width=\textwidth]{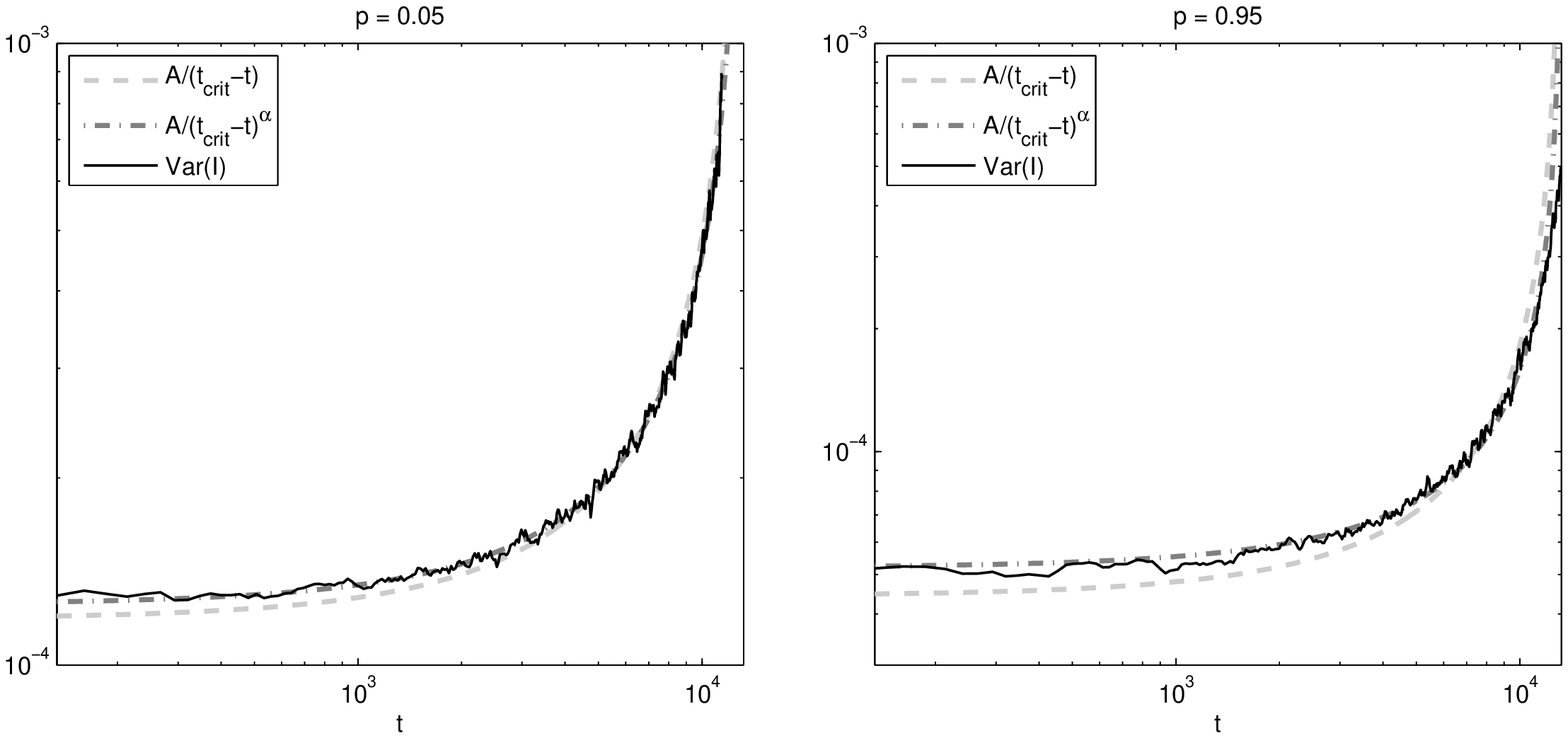}
\caption{Results for the continuous heterogeneous system for two different values of $p$. For $p=0.05$ the best fit over $90\%$ of the time interval was calculated as $\alpha=0.8587$ and $A=0.0074$. For $p=0.95$ these values were $\alpha=0.7885$ and $A=0.0923$. We can see that for $p=0.95$ the variance is visibly below the reference curve with $\alpha=1$ while for $p=0.05$ it still follows this curve quite closely.}\label{figure_6}
\end{figure}

For the heterogeneous system with continuous $\Omega$ we choose $\Omega=[0,1]$ with 
$\mu$ as the Lebesgue measure. At first we again restrict the influence of the heterogeneity 
to the function $f(\omega)$. The choice of the other parameters in unchanged from the discrete 
system. What has to be changed however, is the noise term in the equation. As mentioned above 
we want the noise for different h-states to be dependent on each other. We do this 
by using the first natural approximation of using the same white noise for all h-states, 
i.e.\ $\xi(t,\omega)=\xi(t)$ independent of $\omega$. In Figure \ref{figure_6} we show the 
variance of $I(t)$ against the reference curves for two different values of the parameter $p$. 
In Figure \ref{figure_7} we show, how $p$ influences both $A$ and $\alpha$. The results show
that upon increasing $p$, we first see $\alpha$ increase and $A$ decrease until they stabilize
for larger $p$. We observe that the stabilization approximately happens when the distribution 
$f(\omega)$ starts to have full support on $[0,1]$.\medskip 

\begin{figure}
\centering
\includegraphics[width=\textwidth]{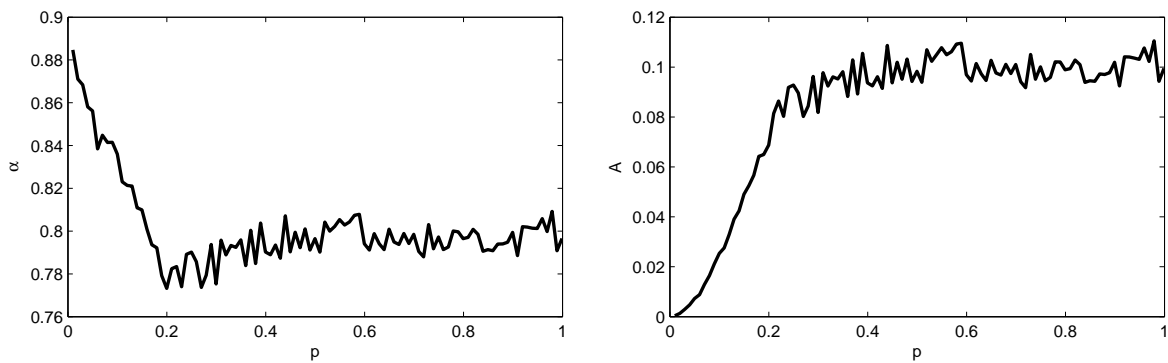}
\caption{This shows the influence of the parameter $p$ on the values $\alpha$ and $A$ of the best fit, calculated 
over $90\%$ of the time interval. In $\alpha$ we see initially a steady decrease until it reaches a constant level. In $A$ we see an initial increase until the values reach a fixed level.  Note that the leveling out both $\alpha$ and $A$ occur for the same values of $p$. Furthermore, by comparing with Figure \ref{figure_5} we see that this coincides with those values of $p$ for which the support of $f(\omega)$ becomes the whole of $\Omega$.}\label{figure_7}
\end{figure}

The last case we are interested in here is to consider a system where $\beta(t,\omega)$ is 
not separable in the sense that it cannot be factored into a product of functions depending
only on $t$ and $\omega$. From the modelling standpoint, this means that the evolution of the 
transmission rate and the heterogeneity in the population interact in a non-trivial way, for 
example, one may consider the situation when a certain population trait amplifies the change in the
transmission rate, while another trait decreases it. As a first benchmark mathematical example,
we simply set
\ba
\dot{\beta}(t,\omega)=\varepsilon (\omega+0.5)t^{\omega-0.5},
\ea
which is solved by $\beta(t,\omega)=\varepsilon t^{\omega+0.5}$. We restrict any further influence 
of $\omega$ to $f(\omega)$. However, we choose $f(\omega)$ slightly differently than before. We 
set
\ba
f(\omega)=\frac{\frac{1}{\sqrt{2\pi}0.1}
\txte^{-\frac{(\omega-\mu)^2}{2*0.1^2}}}{\displaystyle
\int\limits_{0}^1\frac{1}{\sqrt{2\pi}0.1}\txte^{-\frac{(\zeta-\mu)^2}{2*0.1^2}}\dd\zeta},
\ea
i.e.\ a normal distribution with mean $\mu$ and a standard deviation of $0.1$. We let $\mu$ 
vary in $[0,1]$. All other parameters are the same as before. Figure \ref{figure_8} shows, 
how $\mu$ influences $A$ and $\alpha$ as calculated from an aggregation of $100$ simulations 
and fitted over $90\%$ of the time interval. We observe a very strong trend in the crucial 
exponent $\alpha$, which decreases as the mean $\mu$ of $f(\omega)$ is increased.

\begin{figure}
\includegraphics[width=\textwidth]{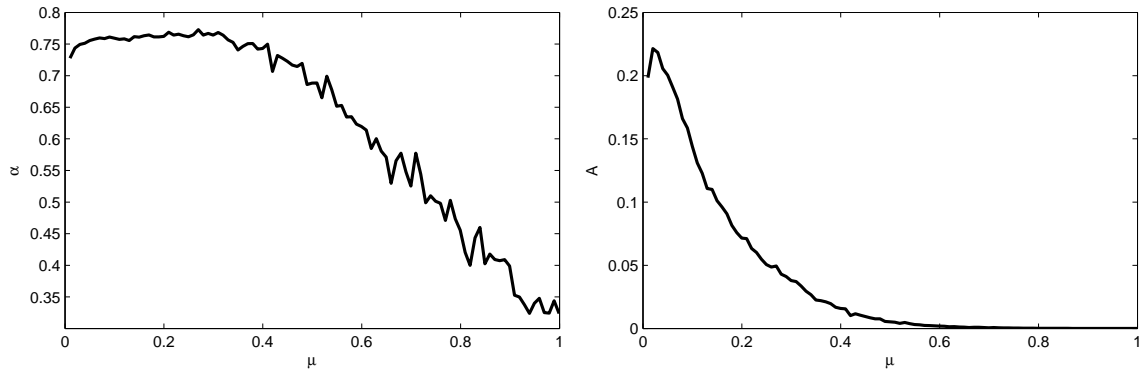}
\caption{We see the influence of the parameter $\mu$ on the values $\alpha$ and $A$ of the best fit, calculated over $90\%$ of the time interval. With increasing $\mu$ both $\alpha$ and $A$ decrease significantly.}\label{figure_8}
\end{figure}

\section{Explanations}
\label{Expl}

In this section we give some explanations, formal or heuristic, for the effect that are observable 
in our simulations

\subsection{Homogeneous system}
\label{Ex_hom}

The first effect we want to explain is the influence of the cut off on the homogeneous system. Since 
the steady state solution $I(t)=0$ is asymptotically stable and the added white noise always has 
an expected value of $0$, in the system without cut off $I(t)$ fluctuates around $0$. Once we introduce 
the cut off $I(t)$ can no longer fluctuate freely. This introduces a bias in the positive direction. 
That is, a sample path $I(t)$ is free to change upwards but we stop it when it changes too far downwards. 
This results in the averaged path being strictly positive (see Figure \ref{figure_9}). Another effect is 
that because we restrict the fluctuations of the white noise we decrease the variance of the resulting
stochastic process $I(t)$. This can be seen by comparing Figures \ref{figure_1} and \ref{figure_2}. 
Finally, in the system without cut off the variance increases at a certain rate. In the system with 
cut off this increase is still present, but we also have a second effect at work. Due to the 
fact that the averaged path also increases, each individual sample path has, as it were, more space 
to fluctuate in, as a downwards deviation from the average path can now be bigger than before 
without hitting $0$. Thus in addition to the usual increase in the variance there is also a 
decrease of the restriction we place on the variance. Therefore, the increase of the variance is 
steeper in the system with cut off. This steeper increase is translated into a decrease of $\alpha$.

\begin{figure}
\includegraphics[width=\textwidth]{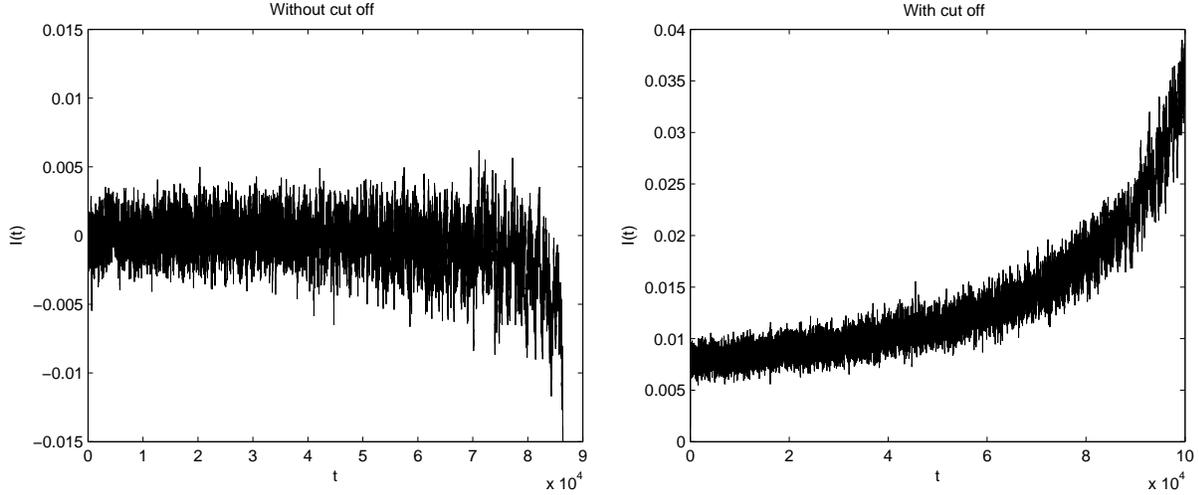}
\caption{Averaged path of the homogeneous system, averaged over 100 simulation, both with and without cut off. The path without cut off eventually tends towards $-\infty$ as it drops below the unstable branch of the transcritical bifurcation.}\label{figure_9}
\end{figure}

\subsection{Discrete heterogeneous system}
\label{Ex_dh}

We now want to analyse the observed changes in the heterogeneous system. We first look at the 
case where $\Omega$ is discrete. Recall that we used
\ba
f(\omega)=\frac{1}{\sqrt{2\pi}\theta}\txte^{-\frac{(\omega-0.5)^2}{2\theta^2}}\frac{1}{C}.
\ea
with
\ba
C=\frac{1}{n}\sum\limits_{i=1}^{n}\frac{1}{\sqrt{2\pi}\theta}
\txte^{-\frac{\left(\frac{i}{n-1}-0.5\right)^2}{2\theta^2}}.
\ea
In Figure \ref{figure_10} we show, how this normalisation constant $C$ changes with $n$. Since 
$C$ is increasing in $n$ we have that, heuristically, for a fixed $\omega\in\Omega$ the value 
$f(\omega)$ decreases. A more rigorous way to state this is to say that if $\omega$ is in 
$\Omega$ for a discretisation level $n_1$ and for a level $n_2$ with $n_1<n_2$, then $f(\omega)$ is 
smaller for $n_2$. Now note that due to the fact that we have chosen most of our parameters 
independent of $\omega$, the linearisation of $\dot{I}(t,\omega)$ is given by
\ba
\dot{\cI}(t,\omega)&=\beta(t) f(\omega)\cI(t)-\gamma \cI(t,\omega)+\sigma\xi(t,\omega).
\ea
Thus, if $f(\omega)$ becomes smaller then $I(t,\omega)$ becomes more ``rigid'', i.e.\ it fluctuates 
less, which results in smaller value of $A$. But this in turn also means that as $R_0$ approaches 
$1$ the additional freedom to fluctuate increases. This results in a bigger increase in the 
variance of $I(t)$ and thus a smaller value of $\alpha$. Both of these effects are visible in 
Figure \ref{figure_4}. Furthermore, by comparing Figures \ref{figure_4} and \ref{figure_10} we see that the levelling out of $\alpha$ and $A$ coincides with the levelling out of $C$.

\begin{figure}
\centering
\includegraphics[width=0.5\textwidth]{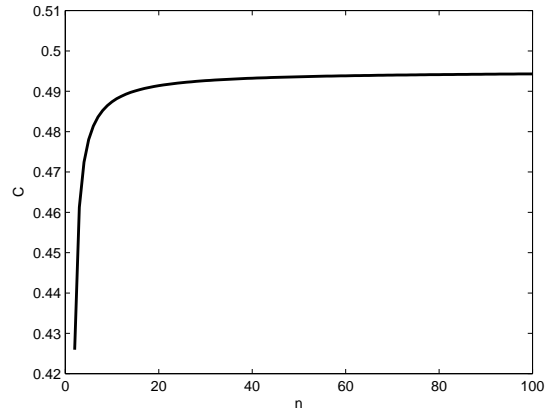}
\caption{The normalisation constant $C$ for $n=2,\dots,100$.}\label{figure_10}
\end{figure}

\subsection{Continuous heterogeneous system}
\label{Ex_ch}

For the heterogeneous system with continuous $\Omega$ we note that by definition we always 
have $I(t,\omega)\in[0,f(\omega)]$. If the parameter $p$ is big enough then $f(\omega)$ is 
large enough for all $\omega$ so that the upper bound is not important due to the fact that 
it is never reached. If $p$ is small however, then $f(\omega)$ also becomes small for some 
$\omega$. Thus we not only have a cut off at $0$ but also at $f(\omega)$. Thus, for small 
$p$ the variance is even more restricted. Also for these $\omega$ a rise of the average 
path will not result in more freedom in its fluctuation due to the restriction above by 
$f(\omega)$. Only when $p$ increases and the upper bound $f(\omega)$ becomes less and less 
important, then the increase of the variation is aided by a increased freedom to fluctuate, 
which leads to lower values of $\alpha$. In Figure \ref{figure_7} we see exactly this behaviour. Since these changes in $A$ and $\alpha$ depend solely on these cut off effects we expect that 
they vanish if we make the same simulations for the system without cut off. The results of 
such a simulation can be seen in Figure \ref{figure_11}, where indeed $p$ has no discernible 
influence on $A$ or $\alpha$.

\begin{figure}
\centering
\includegraphics[width=\textwidth]{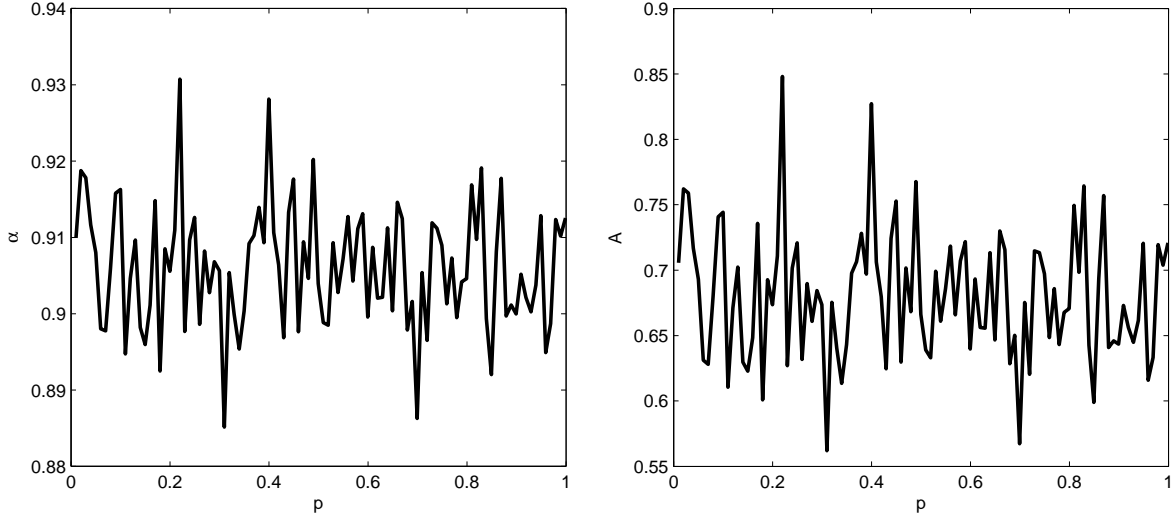}
\caption{We show, dependent on $p$, the change in the values $\alpha$ and $A$ of the best fit, calculated 
over $80\%$ of the time interval, for the heterogeneous system without cut off. There is no discernible influence of $p$ present.}
\label{figure_11}
\end{figure}

\subsection{Non-separable $\beta(t,\omega)$}
\label{Ex_ns}

In order to explain our observations of the system where $\beta(t,\omega)$ is not 
separable we first look at the linearisation of the equations. We assume that all 
functions in our equations are in $L^2(\Omega)$. We can write for the deterministic 
system (\ref{Eq_I})
\ba
\dot{I}(t,\omega)&=F(I(t,\omega))
\ea
with 
\ba
F(I(t,\omega))&=\beta(\omega)\int_\Omega q(\omega)I(t,\omega)
\dd\omega(f(\omega)-I(t,\omega))-\gamma(\omega)I(t,\omega)
\ea
The Fr\'echet-derivative of $F$, evaluated at $I^*$ and applied to $\zeta(\omega)$, is given by
\ba
\left[\frac{\dd F}{\dd I}(I^*)\right]\zeta(\omega)&=
\beta(\omega)\int_\Omega q(\omega)\zeta(\omega)\dd\omega~(f(\omega)-I^*)-
\beta(\omega)\int_\Omega q(\omega)I^*\dd\omega~\zeta(\omega)-\gamma(\omega)\zeta(\omega).
\ea
We define a linear operator $T$ by $T\cI(t,\omega)=\left[\frac{\dd F}{\dd I}(0)\right]\cI(t,\omega)$. 
In particular, the equation linearised at $0$ reads as
\ba
\dot{\cI}(t,\omega)&=T\cI(t,\omega)=f(\omega)\beta(\omega)
\int_\Omega q(\omega)\cI(t,\omega)\dd\omega-\gamma(\omega)\cI(t,\omega).
\ea
We are interested in the spectrum of the operator $T$. We consider this operator on the space 
$X=\{\zeta\in L^2(\Omega):\zeta(\omega)\in[0,f(\omega)]\}$, i.e.\ the subset of $L^2(\Omega)$ that 
consists of the points which are possible states of our system. A point $\lambda\in\C$ is an 
eigenvalue of $T$ if and only if there exists an eigenvector $\zeta\in X$ such that $T\zeta-\lambda\zeta=0$. 
This equation in its longer form is
\ba
f(\omega)\beta(\omega)\int_\Omega q(\omega)\zeta(\omega)\dd\omega
-\gamma(\omega)\zeta(\omega)-\lambda\zeta(\omega)=0.
\ea
We can rearrange this to get 
\ba
\zeta(\omega)=f(\omega)\frac{\beta(\omega)}{\gamma(\omega)+\lambda}\int_\Omega q(\omega)\zeta(\omega)\dd\omega.
\ea
Plugging this into the above equation yields
\ba
0&=f(\omega)\beta(\omega)\int_\Omega q(\omega)f(\omega)\frac{\beta(\omega)}{\gamma(\omega)
+\lambda}\dd\omega\int_\Omega q(\omega)\zeta(\omega)\dd\omega-f(\omega)\beta(\omega)\int_\Omega 
q(\omega)\zeta(\omega)\dd\omega\\
&=f(\omega)\beta(\omega)\int_\Omega q(\omega)\zeta(\omega)\dd\omega\left(\int_\Omega 
q(\omega)f(\omega)\frac{\beta(\omega)}{\gamma(\omega)+\lambda}\dd\omega-1\right)
\ea
An eigenvalue $\lambda$ of $T$ must therefore satisfy
\bea\label{Eq_EV}
\int_\Omega q(\omega)f(\omega)\frac{\beta(\omega)}{\gamma(\omega)+\lambda}\dd\omega=1.
\eea
or
\bea\label{Eq_alt_EV}
(\gamma(\omega)+\lambda)\zeta(\omega)=0\quad\text{and}\quad \int_\Omega q(\omega)\zeta(\omega)\dd\omega=0.
\eea
Note that any $\lambda$ that satisfies the first equation in (\ref{Eq_alt_EV}) is negative as $\gamma(\omega)$ is strictly positive. Furthermore, due to $q(\omega)$ being a positive function, any eigenvector to fulfil the second equation in (\ref{Eq_alt_EV}) has to be negative somewhere. The domain $X$ which we consider for $T$ does therefore not contain any eigenvectors satisfying this equation. For these reasons we consider only equation (\ref{Eq_EV}) to be relevant for our considerations. This equation has a unique solution. Note that for $\lambda=0$ the left hand side is 
exactly $R_0$. In particular, $\lambda$ is positive if $R_0>1$ and negative if $R_0<1$.
Note that if we assume in our calculations that $\gamma(\omega)$ is independent of $\omega$ then 
we can rearrange the equation (\ref{Eq_EV}) to identify $\lambda$ as\footnote{This is an example why we only consider equation (\ref{Eq_EV}): if $\gamma(\omega)$ is constant then (\ref{Eq_alt_EV}) has exactly one solution $\lambda_2=-\gamma$. This eigenvalue is always smaller than $\lambda$ and does therefore not concern us.}
\ba
	\lambda=\int_\Omega q(\omega)f(\omega)\beta(\omega)\dd\omega-\gamma.
\ea
If we now assume that $\beta(t,\omega)$ is a time dependent slow variable, then we can expect 
that this equation approximately describes the evolution of $\lambda$. In particular, if 
$\beta(t,\omega)$ is separable, $\beta(t,\omega)=\beta_0(t)\beta(\omega)$, then 
$\int_\Omega q(\omega)f(\omega)\beta(\omega)\dd\omega$ is a constant $\kappa$ and we get
\ba
\lambda(t)=\beta_0(t)\kappa-\gamma.
\ea
We know that $\lambda(t)$ is the exponential rate with, which the quasi-stationary system 
($\varepsilon=0$) would go to $0$. Hence, for negative $\lambda(t)$, the smaller it is the more 
``rigid'' the system is. If $\beta_0(t)$, and thus $\lambda(t)$, is increasing fast near the 
critical point then it is tightly locked to $0$ until shortly before $t_{crit}$. Therefore, we 
expect a sharp increase in the variation close to $t_{crit}$ and thus a low $\alpha$. We show 
this effect for the homogeneous system in Figure \ref{figure_12}.

\begin{figure}
\centering
\includegraphics[width=\textwidth]{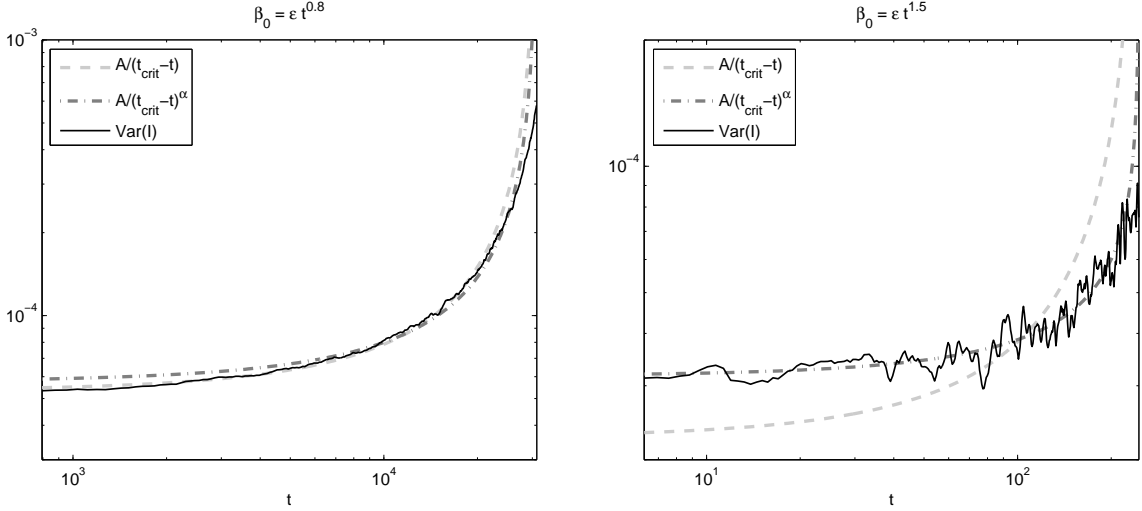}
\caption{The homogeneous system for different $\beta_0(t)$, with best fit over $90\%$ of the time interval. While for $\beta_0(t)=\varepsilon t^{0.8}$ the variance is still in the vicinity of the reference curve with the theoretical exponent $\alpha=1$ (although visibly below it), for $\beta_0(t)=\varepsilon t^{1.5}$ these curves are markedly different. This can also be seen in the value $\alpha$ of the best fit. In the former case it is $\alpha=0.8435$ while for the latter we get $\alpha=0.3795$.}
\label{figure_12}
\end{figure}

In our simulation for the heterogeneous system we achieve the same effect by changing 
the distribution $f(\omega)$. Recall that we used $\beta(t,\omega)=\varepsilon t^{\omega+0.5}$. 
Thus, for $\omega=0$ the increase is as the square root of $t$ while for $\omega=1$ it is 
polynomial. With the parameter $\mu$ we can control, which increase is dominant. If $\mu$ is 
small then $f(\omega)$ is concentrated on those $\omega$ for which $\beta(t,\omega)
\approx\varepsilon t^{0.5}$. Thus it grows slowly and we expect a higher $\alpha$. Also the 
system is less ``rigid" and allows for a higher overall variance in $I(t)$ and thus larger 
$A$. As $\mu$ increases, so does the derivative of $\lambda(t)$ and we expect a more rigid 
system (hence smaller $A$) and a faster increase of the variance near the critical point 
(smaller $\alpha$). Both of these behaviours can be seen in Figure \ref{figure_8}. In Figure 
\ref{figure_13} we show how $\lambda(t)$ behaves for different choices of $\mu$.

\begin{figure}
\includegraphics[width=\textwidth]{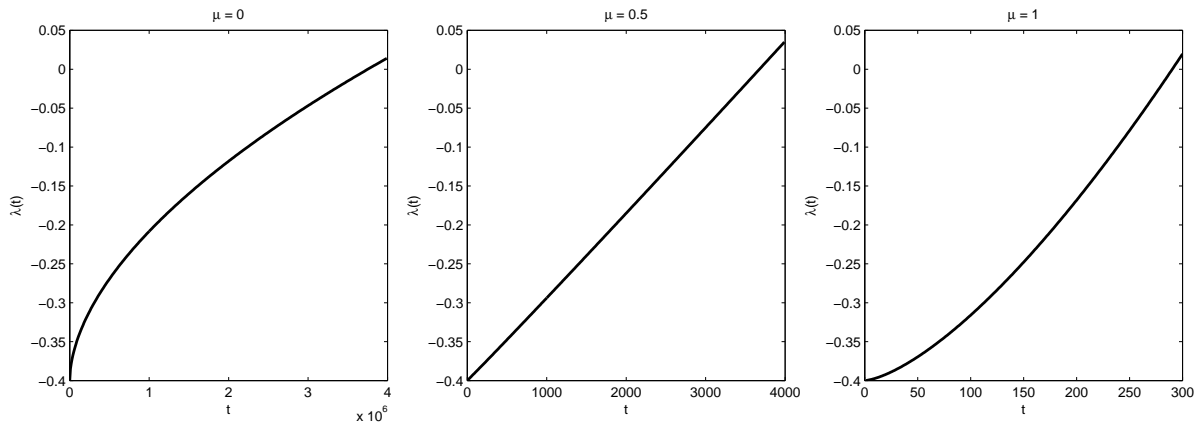}
\caption{The function $\lambda(t)$ for different choices of $\mu$. We can see that for $\mu=0$ the function $\lambda(t)$ is concave and for $\mu=1$ it is convex. For the intermediate value $\mu=0.5$ it is approximately linear. Furthermore we see a significant difference in the time it takes for $\lambda(t)$ to reach $0$.}
\label{figure_13}
\end{figure}

\section{Outlook}
\label{Outlook}

In this paper we have provided new insights on qualitative persistence and quantitative 
non-persistence of various dynamical phenomena in an SIS-model with heterogeneous populations.
The main conclusions are that one can expect a generic dynamical structure of a disease-free 
and endemic state, separated by a transition at $R_0$, to persist. However, the classical 
warning signs for tipping points have to be re-considered carefully in heterogeneous epidemic
models. In particular, we observed that the scaling law exponent for the inverse power-law increase 
of the variance decreases and in many cases lies below the theoretically predicted values of the
homogeneous population system. This means that using an extrapolation procedure with fixed exponent 
to predict the region, where the practical $R_0$-value lies, may not give the correct epidemic
threshold.\medskip

Since, this work is one of the first investigations of warning signs in
heterogeneous population models, it is clear that many open questions remain. Here we shall just
mention a few of these. From a mathematical perspective, it would be natural to ask for a full
analytical description of phenomena arising near bifurcation points for heterogeneous stochastic 
fast-slow systems. There are basically no results in this direction available yet, although 
recent significant progress in mathematical multiscale dynamics may suggest that a (partial) analysis
should be possible \cite{KuehnCT2}. From a biological and epidemic-modelling perspective, 
it would be interesting to compare different classes of models to the fast-slow heterogeneous stochastic 
SIS model we considered with a view towards heterogeneity, epidemic thresholds and warning 
signs for critical transitions. For example, this could include SIR models \cite{HL,N2}, adaptive network 
dynamics \cite{GDB,MPV,SS}, and stochastic partial differential equations \cite{A,KuehnFKPP}. 

Of course, many other extensions of the model, for example demographic changes, could also 
influence the behaviour. A focus on quantitative scaling laws could shed new light on 
which models are most appropriate for certain disease outbreaks, when results are compared with 
data.\medskip

Of course, our study here only carries out a few important baseline steps to achieve these future
goals. Nevertheless, it provides clear evidence for the need to further investigate the interplay
between various effects such as parameter drift, noise, and heterogeneity in the context of 
biological models, which exhibit bifurcation phenomena of high practical and social relevance.

\appendix

\section{The convergence in mean}
\label{ap:lemma}

Here we prove the auxillary result Lemma \ref{lem:convergence}, which shows that for the deterministic heterogenous 
SIS model we study, a suitable weighted mean of the infected population 
$J(t):=\int_{\Omega}q(\omega)I(t,\omega)~\txtd \omega$ has a well-defined limit.

\begin{proof}{ of Lemma \ref{lem:convergence}.}
We employ the same notation as in the proof of Theorem \ref{Th_stab}. In addition, define 
$J^*=\limsup J(t)$ and $J_*=\liminf J(t)$. Assume that $J(t)$ does not converge, 
then $J^*-J_*>0$. In the following five steps we lead this assumption to a contradiction.

\bino
Step 1: Define 
\ba
h(J(t),\omega) = \frac{ f(\omega)(\beta(\omega)J(t)+\eta(\omega))}{\beta(\omega)J(t)+\eta(\omega)+\gamma(\omega)}.
\ea
We get this function by setting $\dot I(t,\omega)=0$ in (\ref{Eq_I}) and solving for $I(t,\omega)$.
Further define $\Omega_f=\{\omega\in\Omega:f(\omega)>0\}$. Obviously $\Omega\backslash\Omega_f$ is of no interest as $I(t,\omega)=0$ there. Note that $\frac{\partial}{\partial J(t)}h(J(t),\omega)=\frac{f(\omega)\beta(\omega)\gamma(\omega)}{(b(\omega)J(t)+\eta(\omega)+\gamma(\omega))^2}>0$ on $\Omega_f$ such that $\frac{\ddd}{\ddd t} h(J(t),\omega)\gtrless0\Leftrightarrow \dot{J}(t)\gtrless0$. This also shows that $h(J(t),\omega)$ is monotone in $J(t)$. We have
\ban\label{I_der_2}
	\dot{I}(t,\omega)\!&=(\beta(\omega)J(t)+\eta(\omega))f(\omega)-(\beta(\omega)J(t)+\eta(\omega)+\gamma(\omega))I(t,\omega)\\
		&=(\beta(\omega)J(t)\!+\!\eta(\omega))f(\omega)\!-\!(\beta(\omega)J(t)\!+\!\eta(\omega)\!+\!\gamma(\omega))\left(h(J(t),\omega)\!+\!I(t,\omega)\!-\!h(J(t),\omega)\right)\\
		&=(\beta(\omega)J(t)+\eta(\omega)+\gamma(\omega))\left(h(J(t),\omega)-I(t,\omega)\right)
\ean
Note that we get
\ban \label{ind_Ider}
\dot{I}(t,\omega)\lessgtr0\Longleftrightarrow  I(t,\omega)\gtrless h(J(t),\omega).
\ean
This proves one of the claims in Lemma \ref{lem:convergence}. Using (\ref{I_der_2}) we get
\ban \label{h_der_est}
\left|\dt \right.&\left.\vphantom{\dt}\!\!h(J(t),\omega)\right|=\left|\frac{\partial}{\partial J(t)}h(J(t),\omega)\dot{J}(t)\right|=\left|\frac{\partial}{\partial J(t)}h(J(t),\omega)\int_\Omega q(\omega)\dot{I}(t,\omega)\dd\omega\right|\\
		&=\left|\frac{f(\omega)\beta(\omega)\gamma(\omega)}{(b(\omega)J(t)+\eta(\omega)+\gamma(\omega))^2}\!\int_\Omega q(\omega)(\beta(\omega)J(t)\!+\!\eta(\omega)\!+\!\gamma(\omega))\left(h(J(t),\omega)\!-\!I(t,\omega)\right)\!\dd\omega\right|\\
		&\leq\frac{f(\omega)\beta(\omega)\gamma(\omega)}{(b(\omega)J(t)\!+\!\eta(\omega)\!+\!\gamma(\omega))^2}\!\int_\Omega q(\omega)(\beta(\omega)J(t)\!+\!\eta(\omega)\!+\!\gamma(\omega))\left|h(J(t),\omega)\!-\!I(t,\omega)\right|\!\dd\omega\\
		&\leq f(\omega)\frac{\beta(\omega)\gamma(\omega)}{(\eta(\omega)+\gamma(\omega))^2}C\int_\Omega q(\omega) f(\omega)\dd\omega\\
		&=f(\omega)\frac{\beta(\omega)\gamma(\omega)}{(\eta(\omega)+\gamma(\omega))^2}C,
\ean
where $C=\sup_{\omega\in\Omega_f}(\beta(\omega)+\eta(\omega)+\gamma(\omega))$.

\bino
Step 2: Define
\ba
	\delta(\omega) &=h(J^*,\omega)-h(J_*,\omega)=f(\omega)\left(\frac{ (\beta(\omega)J^*+\eta(\omega))}{\beta(\omega)J^*+\eta(\omega)+\gamma(\omega)}-\frac{ (\beta(\omega)J_*+\eta(\omega))}{\beta(\omega)J_*+\eta(\omega)+\gamma(\omega)}\right)\\
	&=f(\omega)\frac{\beta(\omega)\gamma(\omega)\left(J^*-J_*\right)}{\left(\beta(\omega)J^*+\eta(\omega)+\gamma(\omega)\right)\left(\beta(\omega)J_*+\eta(\omega)+\gamma(\omega)\right)}.
\ea

For all $\varepsilon>0$ there exist arbitrarily large $t^*$ such that $J(t^*)<J_*+\varepsilon$. We want to give an estimate for $t(\omega)$ such that $h(J(t),\omega)\leq \delta(\omega)/3+h(J_*+\varepsilon,\omega)$ for $t\in(t^*,t(\omega))$. Because of (\ref{h_der_est}) we get
\ba
t(\omega)&\geq\frac{\delta(\omega)}{3f(\omega)\frac{\beta(\omega)\gamma(\omega)}{(\eta(\omega)+\gamma(\omega))^2}C}=\frac{f(\omega)\frac{\beta(\omega)\gamma(\omega)\left(J^*-J_*\right)}{\left(\beta(\omega)J^*+\eta(\omega)+\gamma(\omega)\right)\left(\beta(\omega)J_*+\eta(\omega)+\gamma(\omega)\right)}}{3f(\omega)\frac{\beta(\omega)\gamma(\omega)}{(\eta(\omega)+\gamma(\omega))^2}C}\\
&=\frac{\left(J^*-J_*\right)}{\left(\beta(\omega)J^*+\eta(\omega)+\gamma(\omega)\right)\left(\beta(\omega)J_*+\eta(\omega)+\gamma(\omega)\right)}\frac{(\eta(\omega)+\gamma(\omega))^2}{3C}\\
&\geq\frac{\inf_{\omega\in\Omega_f}\left((\eta(\omega)+\gamma(\omega))^2\right)\left(J^*-J_*\right)}{3C^3}=:\kappa.
\ea
Note that $\kappa>0$ and is independent of $\omega$.

\bino
Step 3: Because of 
\ba
\delta(\omega)\geq f(\omega)\frac{\inf_{\omega\in\Omega_f}(\beta(\omega)\gamma(\omega))\left(J^*-J_*\right)}{C^2}
\ea
we have for every $\varepsilon>0$ a $t_\varepsilon$ such that $h(J(t),\omega)<h(J^*,\omega)+\varepsilon\delta(\omega)/2$ for all $t>t_\varepsilon$. Assume now that $I(t,\omega)>h(J^*,\omega)+\varepsilon \delta(\omega)$. Then using (\ref{I_der_2}) we see that
\ba
	\left|\dot{I}(t,\omega)\right|&=(\beta(\omega)J(t)+\eta(\omega)+\gamma(\omega))\left|h(J(t),\omega)-I(t,\omega)\right|\\
	&\geq\inf_{\omega\in\Omega_f}(\eta(\omega)+\gamma(\omega)) \frac{\varepsilon}{2}\delta(\omega)\geq\frac{\varepsilon}{2}f(\omega)\inf_{\omega\in\Omega_f}\left(\frac{\delta(\omega)}{f(\omega)}\right)\inf_{\omega\in\Omega_f}(\eta(\omega)+\gamma(\omega)).
\ea
From this and (\ref{ind_Ider}) we get that for $t$ large enough we have $I(t,\omega)\leq h(J^*,\omega)+\varepsilon\delta(\omega)$ for all $\omega\in\Omega_f$.

\bino
Step 4: Choose $\varepsilon>0$ small enough such that the two inequalities
\ban \label{proof_ineqs}
	h(J^*,\omega)-\varepsilon\frac{\delta(\omega)}{3}\geq h(J_*+\varepsilon,\omega)+\frac{2}{3}\delta(\omega),\quad 2\varepsilon\leq\frac{\kappa}{3}\inf_{\omega\in\Omega_f}(\eta(\omega)+\gamma(\omega))
\ean
hold true. Now choose a $t^*$ such that $J(t^*)<J_*+\varepsilon$. Let $t^*$ also be large enough such that for all $t\geq t^*$ we have 
\ban \label{proof_tineq}
I(t,\omega)\leq h(J^*,\omega)+\varepsilon\frac{\delta(\omega)}{3}.
\ean
For every $\omega\in\Omega_f$ and for $t\in(t^*,t^*+\kappa)$ where $I(t,\omega)\geq h(J^*,\omega)-\varepsilon\delta(\omega)/3$ we have because of the first inequality in (\ref{proof_ineqs}) that $\left|I(t,\omega)-h(J(t),\omega)\right|\geq\delta(\omega)/3$. Thus, using the second inequality in (\ref{proof_ineqs}), we get
\ba
\left|\dot{I}(t,\omega)\right|&=(\beta(\omega)J(t)+\eta(\omega)+\gamma(\omega))\left|h(J(t),\omega)-I(t,\omega)\right|\\
	&\geq \frac{\delta(\omega)}{3}\inf_{\omega\in\Omega_f}(\eta(\omega)+\gamma(\omega))\geq\frac{2\delta(\omega)\varepsilon}{3\kappa}.
\ea
Combining this with (\ref{proof_tineq}) and using (\ref{ind_Ider}) yields
\ba
I(t^*+\kappa,\omega)\leq h(J^*,\omega)-\varepsilon\frac{\delta(\omega)}{3},\quad\omega\in\Omega_f.
\ea

\bino
Step 5: Let $\tau>t^*+\kappa$ be such that $h(J(\tau),\omega)\geq h(J^*,\omega)-\varepsilon\delta(\omega)/3$ for all $\omega\in\Omega_f$ and $J(\tau)>J(t)$ for $t\in(t^*+\kappa,\tau)$. Since $I(t,\omega)$ is increasing if and only if $h(J(t),\omega)>I(t,\omega)$ we have for all $\omega\in\Omega_f$ that $I(\tau,\omega)\leq h(J(\tau),\omega)$. Thus $\dot{I}(\tau,\omega)\geq0$ for all $\omega\in\Omega_f$ and consequently $\dot{J}(\tau)=\int_{\Omega_f}q(\omega)\dot{I}(\tau,\omega)\dd\omega\geq 0$. Therefore, if $I(t,\omega)=h(J(t),\omega)$ for any $t\geq \tau$, we have that $\dot{I}(t,\omega)=0$ while $\dot{J}(t)\geq 0$ and consequently $\frac{\ddd}{\ddd t} h(J(t),\omega)\geq0$. Hence, $I(t,\omega)\leq h(J(t),\omega)$ for all $t>\tau$ and all $\omega\in\Omega_f$. This in turn implies that $J(t)$ is monotonically increasing for $t>\tau$. Thus, $J(t)$ converges in contradiction to our assumption.

\end{proof}


\end{document}